\documentclass[final]{siamltex}

\usepackage{psfrag,cite,amsmath,graphicx,epsfig,latexsym,subfigure,color}
\usepackage{color}
\usepackage{amssymb}

\def\proj{\mbox{proj}}
\def\prox{\mbox{prox}}

\newcommand{\beq}{\begin{equation}}
\newcommand{\eeq}{\end{equation}}
\newcommand{\st}{{\rm s.t.}}

\newcommand{\cC}{{\mbox{ $\mathcal{C}$}}}

\newcommand{\cK}{{\mbox{ $\mathcal{K}$}}}

\def\pn {\par\smallskip\noindent}

\newcommand{\hx}{\hat{x}}
\newcommand{\hy}{\hat{y}}
\newcommand{\tx}{\tilde{x}}
\newcommand{\ty}{\tilde{y}}

\title{Convergence Analysis of Alternating Direction Method of Multipliers for a Family of Nonconvex Problems}
\author{{Mingyi Hong, Zhi-Quan Luo and Meisam Razaviyayn}\thanks{A conference version of the paper has been presented in 2015 International Conference on Acoustics Speech and Signal Processing (ICASSP) \cite{Hong2015icassp}. \newline
M. Hong is with the Department of Industrial and Manufacturing Systems Engineering, Iowa State University, Ames, IA 50011, USA. Email: \texttt{mingyi@iastate.edu}.\newline Z.-Q. Luo is with the Chinese University of Hong Kong, Shenzhen, China and Department of Electrical and Computer Engineering, University of Minnesota, Minneapolis, MN 55455, USA. Email: \texttt{luozq@cuhk.edu.cn} and \texttt{luozq@umn.edu}. \newline
M. Razaviyayn is with the Department of Electrical Engineering, Stanford University. Email: \texttt{meisamr@stanford.edu}\newline
M. Hong is supported by the National Science Foundation (NSF),
grant CCF-1526078, and by the Air Force Office of Scientific Research (AFOSR), grant 15RT0767. Z.-Q. Luo is supported by the National Science Foundation (NSF), grant CCF-1526434.}}
\begin{document}

\maketitle

\begin{abstract}
{\color{black}The alternating direction method of multipliers (ADMM) is widely used to solve large-scale linearly constrained optimization problems, convex or nonconvex, in many engineering fields. However there is a general lack of theoretical understanding of the algorithm when the objective function is nonconvex. In this paper we analyze the convergence of the ADMM for solving certain nonconvex {\it consensus} and {\it sharing} problems. We show that the classical ADMM converges to the set of stationary solutions, provided that the penalty parameter in the augmented Lagrangian is chosen to be sufficiently large. For the sharing problems, we show that the ADMM is convergent regardless of the number of variable blocks. Our analysis does not impose any assumptions on the iterates generated by the algorithm, and is broadly applicable to many ADMM variants involving proximal update rules and various flexible block selection rules.  }
\end{abstract}


\noindent {\bf AMS(MOS) Subject Classifications:}  49, 90.


\section{\color{black}\bf Introduction} \label{sub:intro}
Consider the following linearly constrained (possibly nonsmooth
or/and nonconvex) problem with $K$ blocks of variables $\{x_k\}_{k=1}^{K}$:
\begin{align}
\begin{split}
\min&\quad f(x):=\sum_{k=1}^{K}g_k(x_k)+\ell(x_1, \cdots, x_K)\label{problem:Original}\\
\st&\quad \sum_{k=1}^{K}A_k x_k=q, \; x_k\in X_k, \ \forall~k=1,\cdots, K
\end{split}
\end{align}
where $A_k\in \mathbb{R}^{M\times N_k}$ and $q\in\mathbb{R}^{M}$; $X_k\subset\mathbb{R}^{N_k}$ is a closed convex set; $\ell(\cdot)$ is a smooth (possibly nonconvex) function; each $g_k(\cdot)$ can be either a smooth function, or a convex nonsmooth function. Let us define $A: =[A_1,\cdots, A_k]$.
The augmented Lagrangian for problem \eqref{problem:Original} is given by
\begin{align}
L(x; y)=\sum_{k=1}^{K} g_k(x_k)+\ell(x_1,\cdots, x_K)+\langle y,
q-Ax\rangle+\frac{\rho}{2}\|q-Ax\|^2 \label{AugmentedL},
\end{align}
where $\rho>0$ is a constant representing the primal penalty parameter.

To solve problem \eqref{problem:Original}, let us consider a popular algorithm called the alternating direction method of multipliers (ADMM), whose steps are given below:
\begin{center}
\fbox{
\begin{minipage}{4.9in}
\smallskip
\centerline{\bf Algorithm 0. ADMM for Problem \eqref{problem:Original}}
\smallskip
At each iteration $t+1$, update the primal variables:
\begin{align}
&x^{t+1}_k=\arg\min_{x_k\in X_k} L(x^{t+1}_1,\cdots, x^{t+1}_{k-1}, x_k, x^{t}_{k+1},\cdots, x^{t}_K; y^{t}), \ \forall~k=1,\cdots, K. \label{eq:basic_x_update}
\end{align}
Update the dual variable:
\begin{align}
&y^{t+1}= y^{t}+\rho(q-A x^{t+1}).\label{eq:basic_dual_update}
\end{align}

\end{minipage}
}
\end{center}

The ADMM algorithm was originally introduced in early 1970s \cite{ADMMGlowinskiMorroco,ADMMGabbayMercier}, and has since been studied extensively \cite{Eckstein89, EcksteinBertsekas1992, Glow84, bertsekas97}. Recently it has become widely popular in modern big data related problems arising in machine learning, computer vision, signal processing, networking and so on; see  \cite{BoydADMM, Yin:2008:BIA:1658318.1658320, Yang09TV, zhang11primaldual,Scheinberg10inverse, Schizas08, Feng14, liao14sdn} and the references therein. In practice, the algorithm often exhibits faster convergence than traditional primal-dual type algorithms such as the dual ascent algorithm \cite{bertsekas99, boyd04, Nedic09} or the method of multipliers \cite{bertsekas82}. It is also particularly suitable for parallel implementation \cite{BoydADMM}.

There is a vast literature that applies the ADMM to various problems in the form of \eqref{problem:Original}. Unfortunately, theoretical understanding of the algorithm is still fairly limited. For example, most of its convergence analysis is done for certain special form of problem \eqref{problem:Original} --- the {\it two-block convex separable} problems, where $K=2$, $\ell=0$ and $g_1$, $g_2$ are both convex. In this case, ADMM is known to converge under very mild conditions; see \cite{bertsekas97} and \cite{BoydADMM}. Under the same conditions, several recent works \cite{HeYuan2012, Monteiro13,Davis14} have shown that the ADMM converges with the sublinear rate of $\mathcal{O}(\frac{1}{t})$ or $o(\frac{1}{t})$, and it converges with a rate $\mathcal{O}(\frac{1}{t^2})$ when properly accelerated \cite{goldstein12, goldfarb12}. Reference \cite{ADMMlinearYin} has shown that the ADMM converges linearly when the objective function as well as the constraints satisfy certain additional assumptions. For the {\it multi-block} separable convex problems where $K\ge 3$, it is known that the original ADMM can diverge for certain pathological problems \cite{chen13}. Therefore, most research effort in this direction has been focused on either analyzing problems with additional conditions, or showing convergence for variants of the ADMM; see for example \cite{HongLuo2012ADMM, he:alternating12, chen13, hong13BSUMM, Wang13, han12admm, Deng14Parallel,HeXuYuan2013,lin_ma_zhang2014}. It is worth mentioning that when the objective function is not separable across the variables (e.g., the coupling function $\ell(\cdot)$ appears in the objective), the convergence of the ADMM is still open, even in the case where $K=2$ and $f(\cdot)$ is convex. Recent works of \cite{hong13BSUMM, hong13bcdmm_icassp} have shown that when problem \eqref{problem:Original} is convex but not necessarily separable, and when certain error bound condition is satisfied, then the ADMM iteration converges to the set of primal-dual optimal solutions, provided that the dual stepsize decreases in time. Another recent work in this direction can be found in \cite{Gao15:nonseparate}.

Unlike the convex case, for which the behavior of ADMM has been investigated quite extensively, when the objective becomes nonconvex, the convergence issue of ADMM remains largely open.
Nevertheless, it has been observed by many researchers that the ADMM works extremely well for various applications involving nonconvex objectives, such as the nonnegative matrix factorization \cite{zhang10ADMM_NMF, sun14}, phase retrieval \cite{wen12}, distributed matrix factorization \cite{zhang14}, distributed clustering \cite{Forero11}, sparse zero variance discriminant analysis \cite{Ames13LDA}, polynomial optimization\cite{jiang13ADMM}, tensor decomposition \cite{Liavas14}, matrix separation \cite{Shen:2014}, matrix completion \cite{xu11admm_matrix}, asset allocation \cite{Wen13Risk}, sparse feedback control \cite{lin13} and so on. However, to the best of our knowledge, existing convergence analysis of ADMM for nonconvex problems is very limited ---  all known global convergence analysis needs to impose uncheckable conditions on the sequence generated by the algorithm. For example, references \cite{jiang13ADMM, Shen:2014, xu11admm_matrix, Wen13Risk} show global convergence of the ADMM to the set of stationary solutions for their respective nonconvex problems, by making the key assumptions that the limit points do exist, and that the successive differences of the iterates (both primal and dual) converge to zero. However such assumption is nonstandard and overly restrictive. It is not clear whether the same convergence result can be claimed without making assumptions on the iterates. Reference \cite{zhang10ADMM_QP} analyzes a family of splitting algorithms (which includes the ADMM as a special case) for certain nonconvex quadratic optimization problem, and shows that they converge to the stationary solution when certain condition on the dual stepsize is met. { We note that there has been many recent works proposing new algorithms to solve nonconvex and nonsmooth problems, for example \cite{Razaviyayn12SUM,Ghadimi15acc:nonconvex,Ghadimi14mini,scutari13decomposition,Bolte14}. However, these works do not deal with nonconvex problems with linearly coupling constraints, and their analysis does not directly apply to the ADMM-type methods. }

The aim of this paper is to provide some theoretical justification on the good performance of the ADMM for nonconvex problems. Specifically, we establish the convergence of ADMM for certain types of nonconvex problems including the consensus and sharing problems without making any assumptions on the iterates. Our analysis shows that, as long as the objective functions $g_k$'s and $\ell$ satisfy certain regularity conditions, and the penalty parameter $\rho$ is chosen large enough (with computable bounds), then the iterates generated by the ADMM is guaranteed to converge to the set of stationary solutions. It should be noted that our analysis covers many variants of the ADMM including per-block proximal update and flexible block selection. An interesting consequence of our analysis is that for a particular reformulation of the sharing problem, the {\it multi-block} ADMM algorithm converges, regardless of the convexity of the objective function. Finally, to facilitate possible applications to other nonconvex problems, we highlight the main proof steps in our analysis framework that can guarantee the global convergence of the ADMM iterates \eqref{eq:basic_x_update}--\eqref{eq:basic_dual_update} to the set of stationary solutions.


\section{The Nonconvex Consensus Problem}

\subsection{The Basic Problem}\label{sec:consensus}
Consider the following nonconvex global consensus problem with regularization
\begin{align}\label{eq:consensus}
\begin{split}
\min&\quad f(x):=\sum_{k=1}^{K}g_k(x)+h(x)\\
\st&\quad x\in X
\end{split}
\end{align}
where $g_k$'s are a set of smooth, possibly nonconvex functions, while $h(x)$ is a convex nonsmooth regularization term.  This problem is related to  the convex global consensus problem discussed heavily in \cite[Section 7]{BoydADMM}, but with the important difference that $g_k$'s can be nonconvex.

In many practical applications, $g_k$'s need to be handled by a single agent, such as a thread or a processor. This motivates the following consensus formulation. Let us introduce a set of new variables $\{x_k\}_{k=0}^{K}$, and transform problem \eqref{eq:consensus} equivalently to the following linearly constrained problem
\begin{align}\label{eq:consensus:admm}
\begin{split}
\min&\quad \sum_{k=1}^{K}g_k(x_k)+h(x_0)\\
\st&\quad x_k=x_0,\;\forall~k=1,\cdots,K, \quad x_0\in X.
\end{split}
\end{align}
{We note that after reformulation, the problem dimension is increased by $K$ due to the introduction of auxiliary variables $\{x_1,\cdots, x_K\}$. Consequently, solving the reformulated problem \eqref{eq:consensus:admm} distributedly
may not be as efficient (in terms of total number of iterations required) as applying the centralized algorithms \cite{Razaviyayn12SUM,Ghadimi15acc:nonconvex,Ghadimi14mini,scutari13decomposition,Bolte14} directly to the original problem \eqref{eq:consensus}.  Nonetheless, a major benefit of solving the reformulated problem \eqref{eq:consensus:admm} is the flexibility of allowing each distributed agent to handle a single {\it local} variable $x_k$ and a {\it local} function $g_k$.}

The augmented Lagrangian function is given by
\begin{align}\label{eq:lagrangian:consensus}
L(\{x_k\}, x_0; y)=\sum_{k=1}^{K}g_k(x_k)+h(x_0)+\sum_{k=1}^{K}\langle y_k, x_k-x_0\rangle+\sum_{k=1}^{K}\frac{\rho_k}{2}\|x_k-x_0\|^2.
\end{align}
Note that this augmented Lagrangian is slightly {different from} the one expressed in \eqref{AugmentedL}, as we have used a set of different penalization parameters $\{\rho_k\}$, one for each equality constraint $x_k=x_0$. We note that there can be many other variants of the basic consensus problem, such as the {\it general form consensus optimization}, the {\it sharing} problem and so on. We will discuss some of those variants in the later sections.

\subsection{The ADMM Algorithm for Nonconvex Consensus}\label{sub:consensus_vanilla}

The problem \eqref{eq:consensus:admm} can be solved distributedly by applying the classical ADMM. The details are given in the table below.
\begin{center}
\fbox{
\begin{minipage}{4.9in}
\smallskip
\centerline{\bf Algorithm 1. The Classical ADMM for Problem \eqref{eq:consensus:admm}}
\smallskip
At each iteration $t+1$, compute:
\begin{equation}
{x_0^{t+1}={\rm arg}\!\min_{x_0\in X}L(\{x^{t}_{ k}\}, x_0;  y^{t}).}
\end{equation}
Each node $k$ computes $x_k$ by solving:
\begin{align}
x^{t+1}_k=\arg\min_{x_k} g_k(x_k)+\langle y^{t}_k, x_k-x_0^{t+1}\rangle+\frac{\rho_k}{2}\|x_k-x_0^{t+1}\|^2.
\end{align}
Each node $k$ updates the dual variable:
\begin{align}
y^{t+1}_k=y_{k}^{t}+\rho_k\left(x^{t+1}_k-x_0^{t+1}\right).
\end{align}
\end{minipage}
}
\end{center}
In the $x_0$ update step, if the nonsmooth penalization $h(\cdot)$ does not appear in the objective, then this step can be written as
\begin{align}
x_0^{t+1}={\rm arg}\!\min_{x_0\in X}L(\{x_{ k}^{t}\}, x_0;  y^{t})=\proj_{X}\left[\frac{\sum_{k=1}^{K}\rho_k x_{k}^{t}+\sum_{k=1}^{K}y_{k}^{t}}{\sum_{k=1}^{K}\rho_k}\right].
\end{align}

Note that the above algorithm has the exact form as the classical ADMM described in \cite{BoydADMM}, where the variable $x_0$ is taken as the first block of primal variable, and the collection $\{x_k\}_{k=1}^{K}$ as the second block. The two primal blocks are updated in a sequential (i.e., Gauss-Seidel) manner, followed by an inexact dual ascent step.

{In what follows, we} consider a more general version of ADMM which includes Algorithm 1 as a special case. In particular, we propose a  {\it flexible} ADMM algorithm in which there is a greater flexibility in choosing the order of the update of both the primal and the dual variables. Specifically, we consider the following two types of variable block update order rules: let $k=0,2,...,K$ be the indices for the primal variable blocks $x_0, x_1,x_2,...,x_K$, and let $\cC^{t}\subseteq\{0,1, \cdots, K\}$ denote the set of variables updated in iteration $t$, then
\begin{enumerate}
\item {\it Randomized update rule}: At each iteration $t+1$, a variable block $k$
is chosen randomly with probability $p^{t+1}_k$,
\begin{align}
{\rm Pr}\left(k\in\cC^{t+1}\mid x_0^t, y^t, \{x_{k}^{t}\}\right)=p^{t+1}_k\ge p_{\rm min}>0.
\end{align}
\item {\it Essentially cyclic update rule}: There exists a given period $T\ge 1$ during which each index is updated at least once. More specifically, at iteration $t$, update all the variables in an index set $\cC^t$ whereby
\begin{align}
\bigcup_{i=1}^{T}\cC^{t+i}=\{0,1,\cdots,K\}, \; \forall~t.
\end{align}
We call this update rule a {\it period-$T$} essentially cyclic update rule.
\end{enumerate}

\begin{center}
\fbox{
\begin{minipage}{4.9in}
\smallskip
\centerline{\bf Algorithm 2. The Flexible ADMM for Problem \eqref{eq:consensus:admm}}
\smallskip
{Let $\cC^{1}=\{0,\cdots, K\}$, $t=0, 1, \cdots$.

At each iteration $t+1$, do:\\

{\bf If} $t+1\ge 2$, pick an index set $\cC^{t+1}\subseteq\{0,\cdots,K\}$}.\\

{\bf If} $0\in \cC^{t+1}$, compute:
\begin{equation}\label{eq:x_update}
{x_0^{t+1}={\rm arg}\!\min_{x\in X}L(\{x_{k}^{t}\}, x_0;  y^{t}).}
\end{equation}
{\bf Else} $x_0^{t+1}=x_0^{t}$. \\
{\bf If} $k\ne 0$ and $k\in\cC^{t+1}$, node $k$ computes $x_k$ by solving:
\begin{align}\label{eq:x_k_update}
x^{t+1}_k=\arg\min_{x_k} g_k(x_k)+\langle y^{t}_k, x_k-x_0^{t+1}\rangle+\frac{\rho_k}{2}\|x_k-x_0^{t+1}\|^2.
\end{align}
\; \; Update the dual variable:
\begin{align}\label{eq:y_update}
y^{t+1}_k=y_{k}^{t}+\rho_k\left(x^{t+1}_k-x_0^{t+1}\right).
\end{align}
{\bf Else} $x^{t+1}_k=x^{t}_k$, $y^{t+1}_k=y^{t}_k$.
\end{minipage}
}
\end{center}

{We note that the randomized version of Algorithm 2 is similar to that of the convex consensus algorithms studied in \cite{Wei13} and \cite{chang14}. It is also related to the {\it randomized BSUM-M} algorithm studied in  \cite{hong13BSUMM}. The difference with the latter is that in the randomized BSUM-M, the dual variable is viewed as an additional block that can be randomly picked (independent of the way that the primal blocks are picked), whereas in Algorithm 2, the dual variable $y_k$ is always updated whenever the corresponding primal variable $x_k$ is updated. To the best of our knowledge, the {period-$T$} essentially cyclic update rule is a new variant of the ADMM. }

{Notice that Algorithm 1 is simply the period-1 essentially cyclic rule, which is a special case of Algorithm 2. Therefore we will focus on analyzing Algorithm 2.} To this end, we make the following assumption.

\pn {\bf Assumption A.}
\begin{itemize}
\item[A1.] There exists a positive constant $L_k>0$ such that $$\|\nabla_k g_k(x_k)-\nabla_k g_k(z_k)\|\le L_k \|x_k-z_k\|, \; \forall~x_k,z_k, \; k=1,\cdots, K.$$
    Moreover, $h$ is convex (possible nonsmooth); $X$ is a closed convex set.
\item[A2.] For all $k$, the penalty parameter $\rho_k$ is chosen large enough such that:
\begin{enumerate}
\item {For all} $k$, the $x_k$ subproblem \eqref{eq:x_k_update} is strongly convex with modulus $\gamma_k(\rho_k)$;
\item For all $k$, $\rho_k\gamma_k(\rho_k)> 2L^2_k$ and $\rho_k\ge L_k$.
\end{enumerate}
\item[A3.] $f(x)$ is bounded from below over  $X$, that is,
{$$\underline{f}:=\min_{x\in X}f(x)>-\infty.$$}
\end{itemize}
We have the following remarks regarding to the assumptions made above.
\begin{itemize}
\item As $\rho_k$ inceases, the subproblem \eqref{eq:x_k_update} will be eventually strongly convex with respect to $x_k$. The corresponding strong convexity modulus $\gamma_k(\rho_k)$ is a monotonic increasing function of $\rho_k$. 
\item Whenever {$g_k(\cdot)$} is nonconvex (therefore $\rho_k>\gamma_k(\rho_k)$), the condition $\rho_k\gamma_k(\rho_k)\ge 2L^2_k$ implies $\rho_k\ge L_k$.
\item By construction, $L(\{x_k\}, x_0; y)$ is also strongly convex with respect to $x_0$, with a modulus $\gamma:=\sum_{k=1}^{K}\rho_k$.
\item Assumption A makes no assumption on the {\it iterates} generated by the algorithm. This is in contrast to the existing analysis of the nonconvex ADMM algorithms  \cite{zhang10ADMM_NMF, xu11admm_matrix, jiang13ADMM}.
\end{itemize}

Now we begin to analyze Algorithm 2. We first make several definitions. Let ${t(k)}$ (resp.\ {$t(0)$}) denote the latest iteration index that $x_k$ (resp.\ $x_0$) is updated before iteration $t+1$, i.e.,{
\begin{align}
\begin{split}
{t(k)}&=\max\; \{r\mid {r\le t, k\in \cC^r}\}, \; k=1,\cdots,K,\label{eq:def_tk}\\
{t(0)}&=\max\; \{r\mid {r\le t, 0\in\cC^r}\}.
\end{split}
\end{align}}
{This definition implies that $x_k^{t}=x_k^{t(k)}$ for all $k=0, \cdots, K$.}

{Also define new vectors $\hx_0^{t+1}$, $\{\hx^{t+1}_k\}$,  $\hy^{t+1}$  and { $\{\tx^{t+1}_k\}$,  $\ty^{t+1}$}  by{
\begin{subequations}
\begin{align}
\hx_0^{t+1}&={\rm arg}\!\min_{x_0\in X}L(\{x_{k}^{t}\}, x_0; y^t),\label{eq:hx}\\
\hx_k^{t+1}&=\arg\min_{x_k} g_k(x_k)+\langle y^{t}_k, x_k-\hx_0^{t+1}\rangle+\frac{\rho_k}{2}\|x_k-\hx_0^{t+1}\|^2, \; \forall~k\label{eq:hxk}\\
\hy^{t+1}_k&=y_{k}^{t}+\rho_k\left(\hx^{t+1}_k-\hx_0^{t+1}\right).\label{eq:hy}\\
\tx_k^{t+1}&=\arg\min_{x_k} g_k(x_k)+\langle y^{t}_k, x_k-x_0^{t}\rangle+\frac{\rho_k}{2}\|x_k-x_0^{t}\|^2, \; \forall~k\label{eq:txk}\\
\ty^{t+1}_k&=y_{k}^{t}+\rho_k\left(\tx^{t+1}_k-x_0^{t}\right).\label{eq:ty}
\end{align}
\end{subequations}
In words, $(\hx_0^{t+1},\{\hx^{t+1}_k\}, \hy^{t+1})$ is a ``virtual" iterate assuming that all variables are updated at iteration $t+1$. $\{\tx^{t+1}_k\}, \ty^{t+1}$ is a ``virtual" iterate for the case where $x_0$ is not updated but the rest of variables are updated.
}}

We first show that the size of the successive difference of the dual variables can be bounded above by that of the primal variables.

\begin{lemma}\label{lemma:y1}
Suppose Assumption A holds. Then for Algorithm 2 with either randomized or essentially cyclic update rule, the following are true
\begin{subequations}
\begin{align}
&L^2_k\|x_k^{t+1}-x_{k}^{t}\|^2\ge \|y^{t+1}_k-y_{k}^{t}\|^2,  \; \forall~k=1,\cdots, K, \label{eq:y_difference1}\\
&{L^2_k\|\hx_k^{t+1}-x_{k}^{t}\|^2\ge \|\hy^{t+1}_k-y_{k}^{t}\|^2,  \; \forall~k=1,\cdots, K.}\label{eq:y_difference_rand1}\\
&{L^2_k\|\tx_k^{t+1}-x_{k}^{t}\|^2\ge \|\ty^{t+1}_k-y_{k}^{t}\|^2,  \; \forall~k=1,\cdots, K.}\label{eq:y_difference_rand1:2}
\end{align}
\end{subequations}
\end{lemma}
\begin{proof}
{We will show the first inequality. The second inequality follows a similar line of argument.

To prove \eqref{eq:y_difference1}, first note that the case for $k\notin \cC^{t+1}$ is trivial, as both sides of \eqref{eq:y_difference1} evaluate to zero. Suppose $k\in\cC^{t+1}$. From the $x_k$ update step \eqref{eq:x_k_update} we have the following optimality condition
\begin{align}
\nabla g_k(x_k^{t+1})+y^t_k+\rho_k (x^{t+1}_k-x_0^{t+1})=0,\; \forall~k\in \cC^{t+1}/ \{0\}.
\end{align}
Combined with the dual variable update step \eqref{eq:y_update} we obtain
\begin{align}\label{eq:y_expression1}
\nabla g_k(x_k^{t+1})=-y^{t+1}_k, \; \forall~k\in\cC^{t+1}/ \{0\}.
\end{align}
Combining this with Assumption A1, and noting that for any given $k$, $y_k$ and $x_k$ are always updated in the same iteration, we obtain for all $k\in\cC^{t+1}/ \{0\}$
\begin{align*}
\|y^{t+1}_k-y^t_k\|&=\|y^{t+1}_k-y^{t(k)}_k\|\nonumber\\
&= \|\nabla g_k(x_k^{t+1})-\nabla g_k(x_k^{t(k)})\|\le L_k\|x_k^{t+1}-x^{t(k)}_k\|= L_k\|x_k^{t+1}-x^{t}_k\|.
\end{align*}
The desired result follows.}
\end{proof}

Next, we use \eqref{eq:y_difference1} to bound the difference of the augmented Lagrangian.
\begin{lemma}\label{lemma:L_difference1}
For Algorithm 2 with either randomized or period-T essentially cyclic update rule, we have the following{
\begin{align}\label{eq:L_descent1}
&L(\{x^{t+1}_k\}, x_0^{t+1}; y^{t+1})-L(\{x^{t}_k\}, x_0^{t}; y^{t})\nonumber\\
&\le\sum_{k\ne 0, k\in\cC^{t+1}}\left(\frac{L^2_k}{\rho_k}-\frac{\gamma_k(\rho_k)}{2}\right)\|x^{t+1}_k-x_{k}^{t}\|^2-\frac{\gamma}{2}\|x_0^{t+1}-x_0^t\|^2.
\end{align}}
\end{lemma}

\begin{proof}
We first split the successive difference of the augmented Lagrangian by
\begin{align}\label{eq:successive_L1}
&L(\{x^{t+1}_k\}, x_0^{t+1}; y^{t+1})-L(\{x^{t}_k\}, x_0^{t}; y^{t})\nonumber\\
&=\left(L(\{x^{t+1}_k\}, x_0^{t+1}; y^{t+1})-L(\{x^{t+1}_k\}, x_0^{t+1}; y^{t})\right)\nonumber\\
&\quad+\left(L(\{x^{t+1}_k\}, x_0^{t+1}; y^{t})-L(\{x^{t}_k\}, x_0^{t}; y^{t})\right).
\end{align}
The first term in \eqref{eq:successive_L1} can be bounded by
\begin{align}
&L(\{x^{t+1}_k\}, x_0^{t+1}; y^{t+1})-L(\{x^{t+1}_k\}, x_0^{t+1}; y^{t})\nonumber\\
&=\sum_{k=1}^{K}\langle y_k^{t+1}-y_{k}^{t}, x^{t+1}_k-x_0^{t+1}\rangle\nonumber\\
&\stackrel{\rm (a)}
=\sum_{k\ne 0, k\in\cC^{t+1}}\frac{1}{\rho_k}\|y_k^{t+1}-y_k^{t}\|^2\label{eq:extra}
\end{align}
where in $\rm (a)$ we have use \eqref{eq:y_update}, and the fact that $y_k^{t+1}-y_{k}^{t}=0$ for all variable block $x_k$ that has not been updated (i.e., $k\ne 0, k\notin \cC^{t+1}$).
The second term in \eqref{eq:successive_L1} can be bounded by
\begin{align}
&L(\{x^{t+1}_k\}, x_0^{t+1}; y^{t})-L(\{x^{t}_k\}, x_0^{t}; y^{t})\nonumber\\
&=L(\{x^{t+1}_k\}, x_0^{t+1}; y^{t})-L(\{x^{t}_k\}, x_0^{t+1}; y^{t})+L(\{x^{t}_k\}, x_0^{t+1}; y^{t})-L(\{x^{t}_k\}, x_0^{t}; y^{t})\nonumber\\
&\stackrel{\rm (a)}\le \sum_{k=1}^{K}\left(\left\langle\nabla_{x_k}  L(\{x^{t+1}_k\}, x_0^{t+1}; y^{t}), x^{t+1}_k-x^{t}_k\right\rangle-\frac{\gamma_k(\rho_k)}{2}\|x^{t+1}_k-x_{k}^{t}\|^2\right)\nonumber\\
&\quad\quad\quad+\left\langle\zeta^{t+1}_{x_0}, x_0^{t+1}-x_0^{t}\right\rangle-\frac{\gamma}{2}\|x_0^{t+1}-x_0^t\|^2\nonumber\\
&\stackrel{\rm (b)}= \sum_{k\ne 0, k\in\cC^{t+1}}\left(\left\langle\nabla_{x_k}  L(\{x^{t+1}_k\}, x_0^{t+1}; y^{t}), x^{t+1}_k-x^{t}_k\right\rangle-\frac{\gamma_k(\rho_k)}{2}\|x^{t+1}_k-x_{k}^{t}\|^2\right)\nonumber\\
&\quad\quad\quad+\iota\{0\in\cC^{t+1}\}\left(\left\langle\zeta^{t+1}_{x_0}, x_0^{t+1}-x_0^{t}\right\rangle-\frac{\gamma}{2}\|x_0^{t+1}-x_0^t\|^2\right)\nonumber\\
&\stackrel{\rm (c)}\le -\sum_{k\ne 0, k\in\cC^{t+1}}\frac{\gamma_k(\rho_k)}{2}\|x^{t+1}_k-x_{k}^{t}\|^2-\iota\{0\in\cC^{t+1}\}\frac{\gamma}{2}\|x_0^{t+1}-x_0^t\|^2,\label{eq:extra1}
\end{align}
where in $\rm (a)$ we have used the fact that $L(\{x_k\}, x_0; y)$ is strongly convex w.r.t.\ each $x_k$ and $x_0$, with modulus $\gamma_k(\rho_k)$ and $\gamma$, respectively, and that {$$\zeta^{t+1}_{x_0}\in \partial_{x_0}  L(\{x^{t}_k\}, x_0^{t+1}; y^{t})$$} is some subgradient vector; in $\rm(b)$ we have used the fact that when {$k\notin\cC^{t+1}$} (resp. {$0\notin\cC^{t+1}$}), $x^{t+1}_k=x_{k}^{t}$ (resp. $x_0^{t+1}=x_0^t$), and we have defined $\iota\{0\in\cC^{t+1}\}$ as the indicator function that takes the value $1$ if $0\in\cC^{t+1}$ is true, and takes value $0$ otherwise; in $\rm (c)$ we have used the optimality of each subproblem \eqref{eq:x_k_update} and $\eqref{eq:x_update}$ {(where $\zeta^{t+1}_{x_0}$ is specialized to the subgradient vector that satisfies the optimality condition for problem \eqref{eq:x_update})}.

Combining the above two inequalities \eqref{eq:extra} and \eqref{eq:extra1}, we obtain
\begin{align}
&L(\{x^{t+1}_k\}, x_0^{t+1}; y^{t+1})-L(\{x^{t}_k\}, x_0^{t}; y^{t})\nonumber\\
&\le -\hspace{-0.5cm}\sum_{k\ne 0, k\in\cC^{t+1}}\hspace{-0.5cm}\frac{\gamma_k(\rho_k)}{2}\|x^{t+1}_k-x_{k}^{t}\|^2+\hspace{-0.5cm}\sum_{k\ne 0, k\in\cC^{t+1}}\frac{1}{\rho_k}\|y_k^{t+1}-y_k^{t}\|^2-\iota\{0\in\cC^{t+1}\}\frac{\gamma}{2}\|x_0^{t+1}-x_0^t\|^2\nonumber\\
&\le\sum_{k\ne 0, k\in\cC^{t+1}}\left(\frac{L^2_k}{\rho_k}-\frac{\gamma_k(\rho_k)}{2}\right)\|x^{t+1}_k-x_{k}^{t}\|^2
-\iota\{0\in\cC^{t+1}\}\frac{\gamma}{2}\|x_0^{t+1}-x_0^t\|^2
\nonumber\end{align}
where the last inequality is due to \eqref{eq:y_difference1}. {The desired result is obtained by noticing the fact that when $0\notin \cC^{t+1}$, we have $x^{t+1}_0-x^t_0 = 0$.}
\end{proof}

The above result implies that if the following condition is satisfied:
\begin{align}\label{eq:rho_condition1}
\rho_k\gamma_k(\rho_k)\ge 2L^2_k,\;\forall~k=1,\cdots, K,
\end{align}
then the value of the augmented Lagrangian function will always decrease. Note that as long as $\gamma_k(\rho_k)\ne 0$, one can always find a $\rho_k$ large enough such that the above condition is satisfied, as the left hand side (lhs) of $\eqref{eq:rho_condition1}$ is monotonically increasing w.r.t.\ $\rho_k$, while the right hand side (rhs) is a constant.

Next we show that $L\left(\{x^{t}_k\}, x_0^{t}; y^{t} \right)$ is in fact convergent.
\begin{lemma}\label{lemma:L_bounded1}
Suppose Assumption A is true. Let $\{\{x_{k}^{t}\},x_0^t,y^t\}$ be generated by Algorithm 2 with either the essentially cyclic rule or the randomized rule. {Then the following limit exists and is lower bounded by $\underline{f}$ defined in Assumption A3:
\begin{align}
\lim_{t\to\infty}L(\{x^{t}_k\}, x_0^t, y^t)\ge\underline{f}.
\end{align}
}
\end{lemma}
\begin{proof}
Notice that the augmented Lagrangian function can be expressed as
\begin{align}\label{eq:L_lower_bound1}
&L(\{x^{t+1}_k\}, x_0^{t+1}; y^{t+1})\nonumber\\
&=h(x_0^{t+1})+\sum_{k=1}^{K}\left(g_k(x^{t+1}_k)+\langle y^{t+1}_k, x^{t+1}_k-x_0^{t+1}\rangle+\frac{\rho_k}{2}\|x^{t+1}_k-x_0^{t+1}\|^2\right)\nonumber\\
&\stackrel{\rm (a)}=h(x_0^{t+1})+\sum_{k=1}^{K}\left(g_k(x^{t+1}_k)+\langle \nabla g_k(x^{t+1}_k), x_0^{t+1}-x_k^{t+1}\rangle+\frac{\rho_k}{2}\|x^{t+1}_k-x_0^{t+1}\|^2\right)\nonumber\\
&\stackrel{\rm (b)}\ge h(x_0^{t+1})+\sum_{k=1}^{K}g_k(x_0^{t+1})=f(x_0^{t+1})
\end{align}
where $\rm (b)$ comes from the Lipschitz continuity of the gradient of  $g_k$'s (Assumption A1), and the fact that $\rho_k\ge L_k$ for all $k=1,\cdots,K$ (Assumption A2). To see why $\rm (a)$ is true, we first observe that due to \eqref{eq:y_expression1}, we have for all $k\ne 0$ and $k\in\cC^{t+1}$
\begin{align}
\langle y^{t+1}_k, x^{t+1}_k-x_0^{t+1}\rangle=\langle \nabla g_k(x^{t+1}_k), x_0^{t+1}-x_k^{t+1}\rangle.\nonumber
\end{align}
{For all $k\ne 0$ and $k\notin \cC^{t+1}$, it follows from $x^{t+1}_k=x^t_k=x^{t(k)}_k=x^{t(k)+1}_k$ and
$y^{t+1}_k=y^t_k=y^{t(k)}_k=y^{t(k)+1}_k$ that}
\begin{align}
&\langle y^{t+1}_k, x^{t+1}_k-x_0^{t+1}\rangle=\langle y_k^{{t(k)}+1}, x_k^{{t(k)}+1}-x_0^{t+1}\rangle\nonumber\\
&=\langle \nabla g_k(x_k^{{t(k)}+1}), x_0^{t+1}-x_k^{{t(k)}+1}\rangle= \langle\nabla g_k(x^{t+1}_k), x_0^{t+1}-x_k^{t+1}\rangle.\nonumber
\end{align}
Combining these two cases shows that $\rm (a)$ is true.

Clearly, \eqref{eq:L_lower_bound1} and Assumption A3 together imply that $L(\{x^{t+1}_k\}, x_0^{t+1}; y^{t+1})$ is lower bounded. This combined with \eqref{eq:L_descent1} says that whenever the penalty parameter $\rho_k$'s are chosen sufficiently large (as per Assumption A2), $L(\{x^{t+1}_k\}, x_0^{t+1}; y^{t+1})$ is monotonically decreasing and is convergent. This completes the proof.
\end{proof}

We are now ready to prove our first main result, which asserts that the sequence of iterates generated by Algorithm 2 converges to the set of stationary solution of problem \eqref{eq:consensus:admm}.

\begin{theorem}\label{thm:convergence1}
{Assume that Assumption A is satisfied. Then we have the following
\begin{enumerate}
\item We have $\lim_{t\to\infty}\|x^{t+1}_k-x_0^{t+1}\|=0, \; k=1,\cdots, K$, deterministically for the essentially cyclic update rule and almost surely for the randomized update rule.
\item Let $(\{x^*_k\}, x_0^*, y^*)$ denote any limit point of the sequence
$\{\{x^{t+1}_k\}, x_0^{t+1}, y^{t+1}\}$ generated by Algorithm 2. Then the following statement is true (deterministically for the essentially cyclic update rule and almost surely for the randomized update rule)
\begin{align}
0 &=\nabla g_k(x^*_k)+y^*_k,\quad k=1,\cdots,K.\nonumber\\
x_0^*&\in \arg\min_{x\in X}\;   h(x)+\sum_{k=1}^{K}\langle y^*_k, x^*_k-x\rangle\nonumber\\
x_k^*&=x_0^*, \quad k=1,\cdots, K. \nonumber
\end{align}
That is, any limit point of Algorithm 2 is a stationary solution of problem \eqref{eq:consensus:admm}.
\item {If $X$ is a compact set, then the sequence of iterates generated by Algorithm 2 converges to the set of stationary solutions of problem \eqref{eq:consensus:admm}. That is,
        \begin{align}
    \lim_{t\to\infty}\mbox{\rm dist}\left( (\{x^t_k\}, x_0^t, y^t); Z^* \right)=0,
    \end{align}
     where $Z^*$ is the set of primal-dual stationary solutions of problem \eqref{eq:consensus:admm}; $\mbox{\rm dist}(x; Z^*)$ denotes the distance between a vector $x$ and the set $Z^*$, i.e.,  $$\mbox{\rm dist}(x; Z^*)=\min_{\hat{x}\in Z^*}\|x-\hat{x}\|.$$}

\end{enumerate}
}
\end{theorem}
\begin{proof}
{We first show part (1) of the theorem}. For the essentially cyclic update rule, Lemma \ref{lemma:L_difference1} implies that
\begin{align}
&L(\{x^{t+T}_k\}, x_0^{t+T}; y^{t+T})-L(\{x^{t}_k\}, x_0^{t}; y^{t})\nonumber\\
&\le\sum_{i=1}^{T}\sum_{k\ne 0, k\in\cC^{t+i}}\left(\frac{L^2_k}{\rho_k}-\frac{\gamma_k(\rho_k)}{2}\right)\|x^{t+i}_k-x^{t+i-1}_k\|^2-\frac{\gamma}{2}\|x_0^{t+i-1}-x_0^{t+i}\|^2\nonumber\\
&= \sum_{i=1}^{T}\sum_{k=1}^{K}\left(\frac{L^2_k}{\rho_k}-\frac{\gamma_k(\rho_k)}{2}\right)\|x^{t+i}_k-x^{t+i-1}_k\|^2-
\frac{\gamma}{2}\|x_0^{t+i-1}-x_0^{t+i}\|^2,\nonumber
\end{align}
where the last equality follows from the fact $x^{t+i}_k=x^{t+i-1}_k$ if $k\not\in\cC^{t+i}$ and $k\neq0$.
Using the fact that each index in $\{0,\cdots, K\}$ will be updated at least once during $[t,\;t+T]$, as well as  Lemma \ref{lemma:L_bounded1} and the bounds for $\rho_k$'s in Assumption A2, we have
\begin{align}\label{eq:differenceZero}
\|x_0^{t+1}-x_0^{t(0)}\|\to 0,\quad \|x^{t+1}_k-x_{k}^{t(k)}\|\to 0, \;\forall~k=1,\cdots,K.
\end{align}
By Lemma~\ref{lemma:y1}, we further obtain $\|y^{t+1}_k-y_{k}^{t(k)}\|\to 0$ for all $k=1,2,...,K$.
In light of the dual update step of Algorithm 2, the fact that $\|y^{t+1}_k-y_{k}^{t(k)}\|\to 0$ implies that $\|x^{t+1}_k-x_0^{t+1}\|\to 0$.

For the randomized update rule, we can take the conditional expectation (over the choice of the blocks) on both sides of \eqref{eq:L_descent1} and obtain{
\begin{align}
&\mathbb{E}\left[L(\{x_k^{t+1}\}, x_0^{t+1}; y^{t+1})-L(\{x^{t}_k\}, x_0^{t}; y^{t})\mid \{x^{t}_k\}, x_0^{t}; y^{t}\right]\nonumber\\
&\le\mathbb{E}\left[\sum_{k\ne 0, k\in\cC^{t+1}}\left(\frac{L^2_k}{\rho_k}-\frac{\gamma_k(\rho_k)}{2}\right)\|x^{t+1}_k-x_{k}^{t}\|^2-\frac{\gamma}{2}\|x_0^{t+1}-x_0^t\|^2 \;{\bigg|}\; \{x^{t}_k\}, x_0^{t}; y^{t}\right]\nonumber\\
&{\le}\sum_{k=1}^{K}p_k {p_0}\left(\frac{L^2_k}{\rho_k}-\frac{\gamma_k(\rho_k)}{2}\right)\|{\hx^{t+1}_k}-x_{k}^{t}\|^2-p_0\frac{\gamma}{2}\|{\hx_0^{t+1}}-x_0^{t}\|^2\nonumber\\
&\quad+\sum_{k=1}^{K}p_k {(1-p_0)}\left(\frac{L^2_k}{\rho_k}-\frac{\gamma_k(\rho_k)}{2}\right)\|{\tx^{t+1}_k}-x_{k}^{t}\|^2\nonumber\\
&\le {p^2_{\rm min}}\sum_{k=1}^{K}\left(\frac{L^2_k}{\rho_k}-\frac{\gamma_k(\rho_k)}{2}\right)\|{\hx^{t+1}_k}-x_{k}^{t}\|^2-p_{\rm min}\frac{\gamma}{2}\|{\hx_0^{t+1}}-x_0^{t}\|^2\nonumber
\end{align}}
{where in the last two inequalities, we have used the fact that $\rho_k$'s satisfy Assumption A2, hence $\frac{L^2_k}{\rho_k}-\frac{\gamma_k(\rho_k)}{2}< 0$ for all $k$}; the last inequality follows from the fact that $p_k\ge p_{\rm min}$ for all $k=0,\cdots,K$.
{
Note that by Lemma \ref{lemma:L_bounded1}, $L(\{x_k^{t+1}\}, x_0^{t+1}; y^{t+1})-\underline{f}\ge 0$ for all $t$, where $\underline{f}$ is defined in Assumption A3. Then let us substract both sides of the above inequality by $\underline{f}$, and invoke the Supermartigale Convergence Theorem \cite[Proposition 4.2]{bertsekas96}}. We conclude that $L(\{x^{t+1}_k\}, x_0^{t+1}; y^{t+1})$ is convergent almost surely (a.s.), and that
\begin{align}
\|{\hx_0^{t+1}}-x_0^{t}\|\to 0,\quad \|{\hx^{t+1}_k}-x_{k}^{t}\|\to 0, \;\forall~k=1,\cdots,K, \quad {\rm a.s.}
\end{align}
By Lemma~\ref{lemma:y1}, we further obtain $\|{\hy^{t+1}_k}-y_{k}^{t}\|\to 0, \ {\rm a.s.}$ and for  all $k=1,2,...,K$.
Finally, from the definition of $\hy^{t+1}$, we see that $\|{\hy^{t+1}_k}-y_{k}^{t}\|\to 0$ a.s.\ implies that {$\|\hx_0^{t+1}-\hx^{t+1}_k\|\to 0$} a.s.\  for all $k=1,2,...,K$.

Next we show part (2) of the theorem. For simplicity, we consider only the essentially cyclic rule as the proof for the randomized rule is similar. We begin by examining the optimality condition for the $x_k$ and $x_0$ subproblems at iteration $t+1$. Suppose $k\ne 0,\; k\in\cC^{t+1}$, then we have
\begin{align}
&\nabla g_k(x^{t+1}_k)+y_{k}^{t}+\rho_k(x^{t+1}_k-x_0^{t+1})= 0.
\end{align}
Similarly, suppose $0\in\cC^{t+1}$, {then there exists an $\eta^{t+1}\in \partial h(x_0^{t+1})$ such that
\begin{align}
&\left\langle x-x_0^{t+1}, \eta^{t+1}-\sum_{k=1}^{K}\left(y_{k}^{t}-\rho_k(x_0^{t+1}-x_0^{t})\right)\right\rangle\ge 0, \; \forall~x\in X.\nonumber
\end{align}}
These expressions imply that
\begin{align}
&\nabla g_k(x^{t+1}_k)+y_{k}^{t}+\rho_k(x^{t+1}_k-x_0^{t+1})= 0, \; k\ne 0,\; k\in\cC^{t+1}\nonumber\\
&h(x)-h(x_0^{t+1})+\left\langle x-x_0^{t+1}, \sum_{k=1}^{K}\left(-y_{k}^{t}+\rho_k(x_0^{t+1}-x_0^{t})\right)\right\rangle\ge 0, \; \forall~x\in X,\; {\rm if}\; 0\in\cC^{t+1}.
\end{align}
Using the definition of the essentially cyclic update rule, we have that for all $t$
\begin{align}\label{eq:opt_t}
\begin{split}
&\nabla g_k(x_k^{r(k)})+y_k^{r(k)}= 0, \; {\forall~k\ne 0},\; \mbox{for some}\; {r(k)}\in[t,\;t+T],\\
&h(x)-h(x_0^{r(0)})+\bigg\langle x-x_0^{r(0)}, \sum_{k=1}^{K}\left(-y^{r(0)-1}_k+\rho_k(x_0^{r(0)}-x^{r(0)-1}_k)\right)\bigg\rangle\ge 0, \\
&\quad\quad\quad\quad\quad\quad\quad\quad\quad\quad\quad\quad\forall~x\in X,\; \mbox{for some}\; r(0)\in[t,\;t+T].
\end{split}
\end{align}
Note that $T$ is finite, and that $\|x^{t+1}_k-x^{t}_k\|\to 0$, $\|x_0^{t+1}-x_0^t\|\to 0$ and $\|y_{k}^{t+1}-y_{k}^{t}\|\to 0$, we have
\begin{align}
&\|x_k^{r(k)}-x_k^{t+1}\|\to 0, \; \forall~k, \quad \|x_0^{r(0)}-x_0^{t+1}\|\to 0, \nonumber\\
&\|y_k^{t+1}-y_k^{r(k)}\|\to 0, \quad \|y_k^{t+1}-y_k^{r(0)-1}\|\to 0, \; \forall~k.
\end{align}
Using this result, taking limit for \eqref{eq:opt_t}, and using the fact that $\|x^{t+1}_k-x^{t}_k\|\to 0$, $x_0^{t+1}\to x_0^*$, $x^{t+1}_k\to x_k^*$, $y_k^{t+1}\to y_k^*$ for all $k$, we have
\begin{align}\label{eq:kkt_sat}
&\nabla g_k(x^*_k)+y^*_k=0, \; k=1,\cdots,K\nonumber\\
&h(x)-h(x_0^{*})+\sum_{k=1}^{K}\big\langle x-x_0^{*}, -y^*_k\big\rangle\ge 0,\; \forall x\in X.
\end{align}
Due to the fact that $\|y_k^{t+1}-y_k^t\|\to 0$ for all $k$, we have that the primal feasibility is achieved in the limit, i.e.,
\begin{align}
x^*_k=x_0^*,\; \forall~k=1,\cdots,K.
\end{align}
This set of equalities together with \eqref{eq:kkt_sat} imply
\begin{align}
&h(x)+\sum_{k=1}^{K}\big\langle x^*_k-x, y^*_k\big\rangle-\left(h(x_0^{*})+\sum_{k=1}^{K}\big\langle x^*_k-x_0^{*}, y^*_k\big\rangle\right)\ge 0,\; \forall x\in X.
\end{align}
This concludes the proof of part (2).

{To prove part 3, we first show that there exists a limit point for each of the sequences $\{x_{k}^{t}\}$, $\{x_0^t\}$ and $\{y^t\}$. Let us consider only the essentially cyclic rule. Due to the compactness assumption of $X$, it is obvious that $\{x_0^t\}$ must have a limit point. Also by a similar argument leading to \eqref{eq:differenceZero}, we see that $\|x_{k}^{t}-x_0^t\|\to 0$, thus for each $k$, $x_{k}^{t}$ must also lie in a compact set thus have a limit point.
Note that the Lipschitz continuity of $\nabla g_k$ combined with the compactness of the set $X$ implies that the set $\{\nabla g_k(x) \mid x\in X\}$ is bounded, therefore $\{\nabla g_k(x_k^{t})\}$ is a bounded sequence. Using \eqref{eq:y_expression1}, we conclude that that $\{y^{t}_k\}$ is also a bounded sequence, therefore must have at least one limit point.}

{We prove part 3 by contradiction. Because the feasible set is compact, then $\{x_k^t\}$ lies in a compact set. From the argument in the previous part it is easy to see that $\{x^t_0\}$, $\{y^t\}$ also lie in some compact sets.  Then every subsequence will have a limit point. Suppose that there exists a subsequence $\{x_{k}^{t_j}\}$, $x_0^{t_j}$ and $\{y^{t_j}\}$ such that
\begin{align}
&    (\{{x}^{t_j}_{k}\}, {x}^{t_j}_0, {y^{t_j}})\to (\{\hat{x}_{k}\}, \hat{x}_0, \hat{y})\label{eq:contradiction1}
\end{align}
where $(\{\hat{x}_{k}\}, \hat{x}_0, \hat{y})$ is some limit point, and by part 2, we have $(\hat{x}_{k}, \hat{x}_0, \hat{y})\in Z^*$. By further restricting to a subsequence if necessary, we can assume that $(\hat{x}_{k}, \hat{x}_0, \hat{y})$ is the unique limit point.

Suppose that this sequence does not converge to the set of stationary solutions, i.e.,
\begin{align}
&    \lim_{j\to\infty}\mbox{\rm dist}\left( (\{x^{t_j}_k\}, x_0^{t_j}, y^{t_j}); Z^* \right)=\gamma>0.\label{eq:contradictio2}
    \end{align}
 Then it follows that there exists some $J(\gamma)>0$ such that
$$ \|(\{{x}^{t_j}_{k}\}, {x}^{t_j}_0, {y^{t_j}})-(\{\hat{x}_{k}\}, \hat{x}_0, \hat{y})\|\le \gamma/2,\quad \forall~j\ge J(\gamma).$$
By the definition of the distance function we have
$$\mbox{dist}\left((\{{x}^{t_j}_{k}\}, {x}^{t_j}_0, {y^{t_j}}); Z^*\right)\le \mbox{dist}\left((\{{x}^{t_j}_{k}\}, {x}^{t_j}_0, {y^{t_j}}),(\{\hat{x}_{k}\}, \hat{x}_0, \hat{y})\right).$$
Combining the above two inequalities we must have
$$\mbox{dist}\left((\{{x}^{t_j}_{k}\}, {x}^{t_j}_0, {y^{t_j}}); Z^*\right)\le \gamma/2, \quad \forall~t_j\ge T_j(\gamma).$$
This contradicts to  \eqref{eq:contradictio2}. The desired result is proven.
}
\end{proof}

The analysis presented above is different from the conventional analysis of the ADMM algorithm where the main effort is to bound the distance between the current iterate and the optimal solution set. The above analysis is partly motivated by our previous analysis of the convergence of ADMM for multi-block convex problems, where the progress of the algorithm is measured by the combined decrease of certain primal and dual gaps; see \cite[Theorem 3.1]{HongLuo2012ADMM}. Nevertheless, the nonconvexity of the problem makes it difficult to estimate either the primal or the dual optimality gaps. Therefore we choose to use the decrease of the augmented Lagrangian as a measure of the progress of the algorithm.

{
Next we analyze the iteration complexity of the vanilla ADMM  (i.e., Algorithm 1). To state our result, let us define the {\it proximal gradient} of the augmented Lagrangian function as
\begin{align}
\tilde{\nabla}L(\{x_k\},x_0,y)=\left[\begin{array}{l}
x_0-\prox_{h}\left[x_0-\nabla_{x_0}(L(\{x_k\},x_0,y)-h(x_0))\right]\\
\nabla_{x_1}L(\{x_k\},x_0,y)\\
\quad\quad\quad\vdots\\
\nabla_{x_K}L(\{x_k\},x_0,y)\\
\end{array}
\right]
\end{align}
where $\prox_{h}[z]:=\arg\min_{x} h(x)+\frac{1}{2}\|x-z\|^2$ is the proximity operator. We will use the following quantity to measure the progress of the algorithm
$$P(\{x_k^{t}\}, x^t, y^t):=\|\tilde{\nabla}L(\{x^t_k\},x^t_0,y^t)\|^2+\sum_{k=1}^{K}\|x^{t}_k-x^{t}_0\|^2.$$
It can be verified that if $P(\{x_k^{t}\}, x^t, y^t)\to 0$, then a stationary solution of the problem \eqref{eq:consensus:admm} is obtained. We have the following iteration complexity result.

\begin{theorem}
Suppose Assumption A is satisfied. Let $T(\epsilon)$ denote an iteration index in which the following inequality is achieved
$$T(\epsilon):=\min\left\{t\mid P(\{x_k^{t}\}, x^t, y^t)\le \epsilon, t\ge 0\right\}$$
for some $\epsilon>0$. Then there exists some constant $C>0$ such that
\begin{align}
\epsilon\le \frac{C (L(\{x^1_k\},x^1_0,y^1)-\underline{f})}{T(\epsilon)}.
\end{align}
where $\underline{f}$ is defined in Assumption A3.
\end{theorem}
\begin{proof}
We first show that there exists a constant $\sigma_1>0$ such that
\begin{align}
\|\tilde{\nabla}L(\{x^t_k\},x^t_0,y^t)\|\le \sigma_1 \left(\|x_0^{t+1}-x_0^t\|+\sum_{k=1}^{K}\|x^{t+1}_k-x^t_k\|\right), \; \forall~r\ge 1. \label{eq:sigma}
\end{align}
This proof follows similar steps of \cite[Lemma 2.5]{HongLuo2012ADMM}. From the optimality condition of the $x_0$ update step \eqref{eq:x_update} we have
$$x^{t+1}_0=\prox_h\left[ x^{t+1}_0-\sum_{k=1}^{K}\rho_k\left(x^{t+1}_0-x^t_k-\frac{y^t_k}{\rho_k}\right)  \right].$$
This implies that
\begin{align}
&\|x^t_0-\prox_{h}\left[x^t_0-\nabla_{x_0}(L(\{x^t_k\},x^t_0,y^t)-h(x^t_0))\right]\|\nonumber\\
&=\left\|x^{t}_0-x^{t+1}_0+x^{t+1}_0-\prox_h\left[x^t_0-\sum_{k=1}^{K}\rho_k(x^t_0-x^{t}_k-\frac{y^t_k}{\rho_k})\right]\right\|\nonumber\\
&\le \|x^{t}_0-x^{t+1}_0\|+\Bigg\|\prox_h\left[ x^{t+1}_0-\sum_{k=1}^{K}\rho_k\left(x^{t+1}_0-x^t_k-\frac{y^t_k}{\rho_k}\right)  \right]\nonumber\\
&\quad\quad\quad\quad-\prox_h\left[x^t_0-\sum_{k=1}^{K}\rho_k(x^t_0-x^{t}_k-\frac{y^t_k}{\rho_k})\right]\Bigg\|\nonumber\\
&\le 2\|x^{t+1}_0-x^t_0\|+\sum_{k=1}^{K}\rho_k\|x^t_0-x^{t+1}_0\|\label{eq:prox1}
\end{align}
where in the last inequality we have used the nonexpansiveness of the proximity operator.

Similarly, the optimality condition of the $x_k$ subproblem is given by
$$\nabla g_k(x^{t+1}_k)+\rho_k\left(x^{t+1}_k-x^{t+1}_0+\frac{y^{t}_k}{\rho_k}\right)=0.$$
Therefore we have
\begin{align}
&\|\nabla_{x_k}L(\{x^t_k\},x^t_0,y^t)\|\nonumber\\
&=\|\nabla g_k(x^{t}_k)+\rho_k(x^{t}_k-x^{t}_0+\frac{y^{t}_k}{\rho_k})\|\nonumber\\
&=\left\|\left(\nabla g_k(x^{t}_k)+\rho_k(x^{t}_k-x^{t}_0+\frac{y^{t}_k}{\rho_k})\right)- \left(\nabla g_k(x^{t+1}_k)+\rho_k(x^{t+1}_k-x^{t+1}_0+\frac{y^{t}_k}{\rho_k})\right)\right\|\nonumber\\
&\le (L_k+\rho_k)\|x^t_k-x^{t+1}_k\|+\rho_k\|x^t_0-x^{t+1}_0\|\label{eq:prox2}.
\end{align}
Therefore, combining \eqref{eq:prox1} and \eqref{eq:prox2}, we have
\begin{align}
  \|\tilde{\nabla}L(\{x^t_k\},x^t_0,y^t)\|\le \left(2+\sum_{k=1}^{K}2\rho_k\right)\|x^t_0-x^{t+1}_0\|+\sum_{k=1}^{K}(L_k+\rho_k)\|x^t_k-x^{t+1}_k\|.\label{eq:estimate:nL}
\end{align}
By taking $\sigma_1=\max\left\{(2+ \sum_{k=1}^{K}2\rho_k), L_1+\rho_1, \cdots, L_K+\rho_K\right\}$, \eqref{eq:sigma} is proved.

According to Lemma \ref{lemma:y1}, we have
\begin{align}
\sum_{k=1}^{K}\|x^t_k-x^t_0\| = \sum_{k=1}^{K}\frac{1}{\rho_k}\|y^{t+1}_k-y^t_k\| \le \sum_{k=1}^{K}\frac{L_k}{\rho_k}\|x_k^{t+1}-x^t_k\|.\label{eq:estimate:x}
\end{align}
The inequalities \eqref{eq:estimate:nL} -- \eqref{eq:estimate:x} implies that for some $\sigma_3>0$
\begin{align}
&\sum_{k=1}^{K}\|x^t_k-x^t_0\|^2 + \|\tilde{\nabla}L(\{x^t_k\},x^t_0,y^t)\|^2\nonumber\\
&\le \sigma_3\left(\|x^t_0-x^{t+1}_0\|^2+\sum_{k=1}^{K}\|x^t_k-x^{t+1}_k\|^2\right). \label{eq:total:square}
\end{align}

According to Lemma \ref{lemma:L_difference1}, there exists a constant $\sigma_2 = \min\left\{\{ \frac{\gamma_k(\rho_k)}{2}-\frac{L^2_k}{\rho_k}\}_{k=1}^{K}, \frac{\gamma}{2} \right\}$
such that
\begin{align}
&L(\{x^{t}_k\}, x_0^{t}; y^{t})-L(\{x^{t+1}_k\}, x_0^{t+1}; y^{t+1})\nonumber\\
&\ge\sigma_2\left(\sum_{k=1}^{K}\|x^{t+1}_k-x_{k}^{t}\|^2+\|x_0^{t+1}-x_0^t\|^2\right).\label{eq:estimate:L}
\end{align}

Combining \eqref{eq:total:square} and \eqref{eq:estimate:L} we have
\begin{align}
  &\sum_{k=1}^{K}\|x^t_k-x^t_0\|^2 + \|\tilde{\nabla}L(\{x^t_k\},x^t_0,y^t)\|^2\nonumber\\
  &\le \frac{\sigma_3}{\sigma_2}\left(L(\{x^{t}_k\}, x_0^{t}; y^{t})-L(\{x^{t+1}_k\}, x_0^{t+1}; y^{t+1})\right)\nonumber.
\end{align}
Summing both sides of the above inequality over $t=1,\cdots, r$, we have
\begin{align}
&\sum_{t=1}^{r}\sum_{k=1}^{K}\|x^t_k-x^t_0\|^2 + \|\tilde{\nabla}L(\{x^t_k\},x^t_0,y^t)\|^2\nonumber\\
  &\le \frac{\sigma_3}{\sigma_2}\left(L(\{x^{1}_k\}, x_0^{1}; y^{1})-L(\{x^{r+1}_k\}, x_0^{r+1}; y^{r+1})\right)\nonumber\\
&\le \frac{\sigma_3}{\sigma_2}\left(L(\{x^{1}_k\}, x_0^{1}; y^{1})-\underline{f}\right)\nonumber
\end{align}
where in the last inequality we have used the fact that $L(\{x^{r+1}_k\}, x_0^{r+1}; y^{r+1})$ is decreasing and lower bounded by $\underline{f}$ (cf. Lemmas \ref{lemma:L_difference1}--\ref{lemma:L_bounded1}).

By utilizing the definition of $T(\epsilon)$ and $P(\{x_k^{t}\}, x^t, y^t)$, the above inequality becomes
\begin{align}
&T(\epsilon)\epsilon\le \frac{\sigma_3}{\sigma_2}\left(L(\{x^{1}_k\}, x_0^{1}; y^{1})-\underline{f}\right)
\end{align}
Dividing both sides by $T(\epsilon)$, and by setting $C=\sigma_3/\sigma_2$, the desired result is obtained.
\end{proof}

}

\subsection{The Proximal ADMM}\label{sub:consensus_proximal}
One potential limitation of Algorithms 1 and 2 is the requirement that each subproblem \eqref{eq:x_k_update} needs to be solved {\it exactly}, while in certain practical applications cheap iterations are preferred.
In this section, we consider an important extension of Algorithm 1--2 in which the above restriction is removed.
The main idea is to take a proximal step instead of minimizing the augmented Lagrangian function exactly with respect to each variable block. Like in the previous section, we will analyze a generalized version, termed the {\it flexible} proximal ADMM, where there is more freedom in choosing the update schedules.

%

\begin{center}
\fbox{
\begin{minipage}{4.9 in}
\smallskip
\centerline{\bf Algorithm 3. A Flexible Proximal ADMM for Problem \eqref{eq:consensus:admm}}
\smallskip
At each iteration $t+1$, compute:
\begin{align}\label{eq:x_update_prox}
x_0^{t+1}&={\rm arg}\!\min_{x\in X}\; L(\{x_{k}^{t}\}, x_0;  y^{t}).
\end{align}
Pick a set $\cC^{t+1}\subseteq\{1,\cdots, K\}$.\\
{\bf If $k\in\cC^{t+1}$}, update $x_k$ by solving:
\begin{align}\label{eq:x_k_update_prox}
x^{t+1}_k=\arg\!\min_{x_k} \; \langle\nabla g_k(x_0^{t+1}), x_k-x_0^{t+1}\rangle+\langle y^{t}_k, x_k-x_0^{t+1}\rangle+\frac{\rho_k+L_k}{2}\|x_k-x_0^{t+1}\|^2.
\end{align}
Update the dual variable:
\begin{align}\label{eq:y_update_prox}
y^{t+1}_k=y_k^{t}+\rho_k\left(x^{t+1}_k-x_0^{t+1}\right).
\end{align}
{\bf Else} let $x^{t+1}_k=x_{k}^{t}$, $y^{t+1}_k=y^{t}_k$.
\end{minipage}
}
\end{center}

Notice that the $x_k$ update step is different from the conventional proximal update (e.g., \cite{BoydADMM}). In particular, the linearization is done with respect to $x_0^{t+1}$ instead of $x_k^t$ computed in the previous iteration. This modification is instrumental in the convergence analysis of Algorithm 3.

{Here we use the {\it period-$T$ essentially cyclic rule} to decide the set $\cC^{t+1}$ at each iteration.} We note that there is a slight difference of the update schedule used in Algorithm 3 and Algorithm 2. In Algorithm 3, the block variable $x_0$ is updated {\it in every iteration} while in Algorithm 2 the update of $x_0$ is also governed by block selection rules.

Now we begin analyzing Algorithm 3. 
We make the following assumptions in this section (in addition to Assumption A1 and A3).

\pn {\bf Assumption B.}
For all $k$, the penalty parameter $\rho_k$ is chosen large enough such that:
\begin{align}
\alpha_k&:={\frac{\rho_k-7L_k}{2}}-\left(\frac{4L_k}{\rho^2_k}+\frac{1}{\rho_k}\right)2L_k^2>0\label{eq:beta}\\
\beta_k&:=\frac{\rho_k}{2}-T^2\left(\frac{4L_k}{\rho^2_k}+\frac{1}{\rho_k}\right)8L_k^2>0\label{eq:alpha}\\
\rho_k& \ge 5 L_k, \; k=1,\cdots, K.
\end{align}

Again let ${t(k)}$ denote the last iteration that $x_k$ is updated before $t+1$, i.e.,
{\begin{align}
{t(k)}&=\max\; \{r\mid {r\le t, k\in \cC^r}\}, \; k=1,\cdots,K.
\end{align}}
{Note that we do not need $t(0)$ anymore since $x_0$ is updated in every iteration.} Clearly, we have $x_k^t=x_k^{t(k)}$ and as a result, $y_k^t=y_k^{t(k)}$.
We have the following result.
\begin{lemma}\label{lemma:y2}
Suppose Assumption B and Assumptions A1, A3 are satisfied. Then for Algorithm 3, the following is true for the essentially cyclic block selection rule
\begin{align}
&2L^2_k(4\|x_0^{t+1}-x_0^{{t(k)}}\|^2+\|x^{t+1}_k-x_k^{t}\|^2)\ge \|y^{t+1}_k-y^{t}_k\|^2, \; k=1,\cdots,K\label{eq:y_difference2}.
\end{align}
\end{lemma}
\begin{proof}
Suppose $k\notin \cC^{t+1}$, then the inequality is trivially true, as $y^{t+1}_k=y^{t}_k$.

For any $k\in \cC^{t+1}$, we observe from the update of $x_k$ step \eqref{eq:x_k_update_prox} that the following is true
\begin{align}
\nabla g_k(x^{t+1})+y^{t}_k+(\rho_k+L_k)(x^{t+1}_k-x_0^{t+1})=0, \; k\in\cC^{t+1},
\end{align}
or equivalently
\begin{align}\label{eq:y_expression2}
\nabla g_k(x^{t+1})+L_k(x^{t+1}_k-x_0^{t+1})=-y^{t+1}_k, \; k\in\cC^{t+1}.
\end{align}
Therefore we have, for all $k\in\cC^{t+1}$
\begin{align}
\|y^{t+1}_k-y^{t}_k\|&=\|y^{t+1}_k-y^{{t(k)}}_k\|\nonumber\\
&= \|\nabla g_k(x_0^{t+1})-\nabla g_k(x_0^{{t(k)}})+L_k(x_k^{t+1}-x_0^{t+1})-L_k(x_k^{{t(k)}}-x_0^{{t(k)}})\|\nonumber\\
&=\|\nabla g_k(x_0^{t+1})-\nabla g_k(x_0^{{t(k)}})+L_k(x_k^{t+1}-x_0^{t+1})-L_k(x_k^{t}-x_0^{{t(k)}})\|\nonumber\\
&\le L_k(2\|x_0^{t+1}-x_0^{{t(k)}}\|+\|x_k^{t+1}-x_k^{t}\|)\nonumber
\end{align}
where the last step follows from triangular inequality and the fact $x^t_k=x^{t(k)}_k$ (cf.\ the definition of $t(k)$).
The above result further implies that
\begin{align}
2L^2_k(4\|x_0^{t+1}-x_0^{{t(k)}}\|^2+\|x^{t+1}_k-x_k^{t}\|^2)&\ge \|y^{t+1}_k-y^{t}_k\|^2, \; k=1,\cdots,K
\end{align}
which is the desired result.
\end{proof}

Next, we {upper} bound the successive difference of the augmented Lagrangian. To this end, let us define the following functions
\begin{align}
\ell_k(x_k; x_0^{t+1},y^t)&=g_k(x_k)+\langle y_k^{t}, x_k-x_0^{t+1}\rangle+\frac{\rho_k}{2}\|x_k-x_0^{t+1}\|^2\nonumber\\
u_k(x_k; x_0^{t+1},y^t)&=g_k(x_0^{t+1})+\langle \nabla g_k(x_0^{t+1}), x_k-x_0^{t+1}\rangle\nonumber\\
&\quad +\langle y_k^t, x_k-x_0^{t+1}\rangle+\frac{\rho_k+L_k}{2}\|x_k-x_0^{t+1}\|^2\nonumber.
\end{align}
Using these short-hand definitions, we have
\begin{align}
&L(\{x^{t+1}_k\}, x_0^{t+1}; y^{t})=\sum_{k=1}^{K}\ell_k(x^{t+1}_k; x_0^{t+1}, y^t)\label{eq:L_ell}\\
&x^{t+1}_k=\arg\min_{x_k} u_k(x_k;x_0^{t+1},y^t), \; \forall~k\in\cC^{t+1}.\label{eq:x_k_u_k}
\end{align}

The lemma below bounds the difference between $\ell_k(x^{t+1}_k; x_0^{t+1},y^t)$ and $\ell_k(x_{k}^{t}; x_0^{t+1},y^t)$.
\begin{lemma}\label{lemma:l_difference}
Suppose Assumption A1 is satisfied. Let $\{x^{t}_k,x_0^{t},y^t\}$ be generated by Algorithm 3 with essential cyclic block update rule. Then we have the following
\begin{align}\label{eq:bound_l}
&\ell_k(x^{t+1}_k;x_0^{t+1},y^t)-\ell_k(x^{t}_k;x_0^{t+1},y^t)\nonumber\\
&\le -{\frac{\rho_k-7L_k}{2}}\|x_k^{t+1}-x_k^{t}\|^2+\frac{4L_k}{\rho^2_k}\|y^{t+1}_k-y_{k}^{t}\|^2, \quad k=1,\cdots,K.
\end{align}
\end{lemma}
\begin{proof}
When $k\notin \cC^{t+1}$, the inequality is trivially true. We focus on the case $k\in\cC^{t+1}$.
From the definition of $\ell_k(\cdot)$ and $u_k(\cdot)$ we have the following
\begin{align}\label{eq:LU}
\ell_k(x_k;x_0^{t+1},y^t)\le u_k(x_k;x_0^{t+1},y^t), \; \forall~x_k,\; k=1,\cdots,K.
\end{align}
Observe that when $k\in\cC^{t+1}$, $x^{t+1}_k$ is generated according to \eqref{eq:x_k_u_k}. Due to the strong convexity of $u_k(x_k;x_0^{t+1},y^t)$ with respect to $x_k$, we have
\begin{align}\label{eq:u_k_difference}
u_k(x^{t+1}_k;x_0^{t+1},y^t)-u_k(x^{t}_{k};x_0^{t+1},y^t)\le -{\frac{\rho_k+L_k}{2}}\|x^{t}_{k}-x^{t+1}_k\|^2, \; \forall~k\in\cC^{t+1}.
\end{align}
Further, we have the following series of inequalities
\begin{align}\label{eq:u_k_l_k}
&u_k(x^{t}_{k};x_0^{t+1},y^t)-\ell_k(x^{t}_{k};x_0^{t+1},y^t)\nonumber\\
&=g_k(x_0^{t+1})+\langle \nabla g_k(x_0^{t+1}), x^{t}_{k}-x_0^{t+1}\rangle+\langle y_k^t, x^{t}_{k}-x_0^{t+1}\rangle+\frac{\rho_k+L_k}{2}\|x^{t}_{k}-x_0^{t+1}\|^2\nonumber\\
&\quad \quad -\left(g_k(x^{t}_{k})+\langle y^{t}_{k}, x^{t}_{k}-x_0^{t+1}\rangle+\frac{\rho_k}{2}\|x^{t}_{k}-x_0^{t+1}\|^2\right)\nonumber\\
&=g_k(x_0^{t+1})-g_k(x^{t}_{k})+\langle \nabla g_k(x_0^{t+1}), x^{t}_{k}-x_0^{t+1}\rangle+\frac{L_k}{2}\|x^{t}_{k}-x_0^{t+1}\|^2\nonumber\\
&\le\langle \nabla g_k(x_0^{t+1})-\nabla g_k(x_k^{t}), x^{t}_{k}-x_0^{t+1}\rangle+L_k\|x^{t}_{k}-x_0^{t+1}\|^2\nonumber\\
&\le2L_k\|x^{t}_{k}-x_0^{t+1}\|^2\le 4L_k\left(\|x^{t}_{k}-x^{t+1}_k\|^2+\|x^{t+1}_k-x_0^{t+1}\|^2\right),
\end{align}
where the first two inequalities follow from Assumption A1.
Combining {\eqref{eq:LU} -- \eqref{eq:u_k_l_k}} 
we obtain
\begin{align}
&\ell_k(x^{t+1}_k; x_0^{t+1},y^t)-\ell_k(x^{t}_k; x_0^{t+1},y^t)\nonumber\\
&\le u_k(x^{t+1}_k; x_0^{t+1},y^t)-u_k(x_k^{t}; x_0^{t+1},y^t)+u_k(x_k^{t}; x_0^{t+1},y^t)-\ell_k(x^{t}_k; x_0^{t+1},y^t)\nonumber\\
&\le-{\frac{\rho_k-7L_k}{2}}\|x_k^{t}-x_k^{t+1}\|^2+4L_k \|x^{t+1}_k-x_0^{t+1}\|^2\nonumber\\
&=-{\frac{\rho_k-7L_k}{2}}\|x_k^{t}-x_k^{t+1}\|^2+\frac{4L_k}{\rho_k^2} \|y_k^{t+1}-y_k^{t}\|^2, \; \forall~k\in\cC^{t+1}.\nonumber
\end{align}
The desired result then follows.
\end{proof}

Next, we bound the difference of the augmented Lagrangian function values.
\begin{lemma}\label{lemma:L_difference2}
Assume the same set up as in Lemma~\ref{lemma:l_difference}. Then we have
\begin{align}\label{eq:L_descent2}
&L(\{x^{t+1}_k\}, x_0^{t+1}; y^{t+1})-L(\{x^{1}_k\}, x_0^{1}; y^{1})\nonumber\\
&\le-\sum_{i=1}^{t}\sum_{k=1}^{K}\alpha_k\|x^{i+1}_k-x^i_k\|^2-\sum_{i=1}^{t}\sum_{k=1}^{K}\beta_k\|x_0^{i+1}-x_0^i\|^2
\end{align}
where we $\beta_k$ and $\alpha_k$ are the positive constants defined in \eqref{eq:beta} and \eqref{eq:alpha}.
\end{lemma}
\begin{proof}
We first bound the successive difference $L(\{x^{t+1}_k\}, x_0^{t+1}; y^{t+1})-L(\{x^{t}_k\}, x_0^{t}; y^{t})$. Again we decompose it as in \eqref{eq:successive_L1}, and bound the {resulting} two differences separately.

The first term in \eqref{eq:successive_L1} can be again expressed as
\begin{align}
&L(\{x^{t+1}_k\}, x_0^{t+1}; y^{t+1})-L(\{x^{t+1}_k\}, x_0^{t+1}; y^{t})=\sum_{k=1}^{K}\frac{1}{\rho_k}\|y_k^{t+1}-y_k^{t}\|^2\nonumber.
\end{align}
To bound the second term in \eqref{eq:successive_L1}, we use Lemma \ref{lemma:l_difference}. We use an argument similar to the proof of \eqref{eq:extra1} to obtain
\begin{align}
&L(\{x^{t+1}_k\}, x_0^{t+1}; y^{t})-L(\{x^{t}_k\}, x_0^{t}; y^{t})\nonumber\\
&=L(\{x^{t+1}_k\}, x_0^{t+1}; y^{t})-L(\{x^{t}_k\}, x_0^{t+1}; y^{t})+L(\{x^{t}_k\}, x_0^{t+1}; y^{t})-L(\{x^{t}_k\}, x_0^{t}; y^{t})\nonumber\\
&= \sum_{k=1}^{K}\left(\ell_k(x^{t+1}_k;x_0^{t+1},y^{t})-\ell_k(x^{t}_k;x_0^{t+1},y^{t})\right)+L(\{x^{t}_k\}, x_0^{t+1}; y^{t})-L(\{x^{t}_k\}, x_0^{t}; y^{t})\nonumber\\
&\le -\sum_{k=1}^{K}\left({\frac{\rho_k-7L_k}{2}}\|x^{t+1}_k-x^{t}_{k}\|^2{-}\frac{4L_k}{\rho^2_k}\|y^{t+1}_k-y_{k}^{t}\|^2\right)-\frac{\gamma}{2}\|x_0^{t+1}-x_0^t\|^2
\end{align}
where the last inequality follows from Lemma~\ref{lemma:l_difference} and the strong convexity of $L(\{x^{t}_k\}, x_0; y^{t})$ with respect to the variable $x$ (with modulus $\gamma=\sum_{k=1}^{K}\rho_k$)  at $x_0=x_0^{t+1}$.

Combining the above two inequalities, we obtain
\begin{align}\label{eq:L_difference2}
&L(\{x^{t+1}_k\}, x_0^{t+1}; y^{t+1})-L(\{x^{t}_k\}, x_0^{t}; y^{t})\nonumber\\
&\le \sum_{k=1}^{K}\left(-{\frac{\rho_k-7L_k}{2}}\|x^{t+1}_k-x^{t}_{k}\|^2+
\left(\frac{4L_k}{\rho^2_k}+\frac{1}{\rho_k}\right)\|y^{t+1}_k-y_{k}^{t}\|^2\right)-\frac{\gamma}{2}\|x_0^{t+1}-x_0^t\|^2\nonumber\\
&\stackrel{\rm (a)}\le \sum_{k=1}^{K}\left(-{\frac{\rho_k-7L_k}{2}}\|x^{t+1}_k-x^{t}_{k}\|^2+
\left(\frac{4L_k}{\rho^2_k}+\frac{1}{\rho_k}\right)2L^2_k(4\|x_0^{t+1}-x_0^{{t(k)}}\|^2+\|x^{t+1}_k-x_k^{t}\|^2)\right)\nonumber\\
&\quad\quad-\frac{\gamma}{2}\|x_0^{t+1}-x_0^t\|^2\nonumber\\
&\stackrel{\rm(b)}= -\sum_{k=1}^{K}\left({\frac{\rho_k-7L_k}{2}}-\left(\frac{4L_k}{\rho^2_k}+\frac{1}{\rho_k}\right)2L_k^2\right)\|x^{t+1}_k-x^{t}_{k}\|^2
-\sum_{k=1}^{K}\left(\frac{\rho_k}{2}\right)\|x_0^{t+1}-x_0^t\|^2\nonumber\\
&\quad\quad\quad \quad +\sum_{k=1}^{K}\left(\frac{4L_k}{\rho^2_k}+\frac{1}{\rho_k}\right)8L_k^2\|x_0^{{t(k)}}-x_0^{t+1}\|^2\nonumber\\
&\le -\sum_{k=1}^{K}\left({\frac{\rho_k-7L_k}{2}}-\left(\frac{4L_k}{\rho^2_k}+\frac{1}{\rho_k}\right)2L_k^2\right)\|x^{t+1}_k-x^{t}_{k}\|^2
-\sum_{k=1}^{K}\left(\frac{\rho_k}{2}\right)\|x_0^{t+1}-x_0^t\|^2\nonumber\\
&\quad\quad\quad \quad +\sum_{k=1}^{K}T\left(\frac{4L_k}{\rho^2_k}+\frac{1}{\rho_k}\right)8L_k^2\sum_{i=0}^{\min\{T-1,t-1\}}\|x_0^{t-i+1}-x_0^{t-i}\|^2
\end{align}
where in $\rm (a)$ we have used \eqref{eq:y_difference2}; in $\rm (b)$ we have used the fact that $\gamma=\sum_{k=1}^{K}\rho_k$; in the last inequality we have used the definition of the period-$T$ essentially cyclic update rule which implies that
\begin{align}\nonumber
&\|x_0^{t+1}-x_0^{{t(k)}}\|\le \sum_{i=0}^{\min\{T-1,t-1\}}\|x_0^{t-i+1}-x_0^{t-i}\|\nonumber\\
&\Longrightarrow \|x_0^{t+1}-x_0^{{t(k)}}\|^2\le T\sum_{i=0}^{\min\{T-1,t-1\}}\|x_0^{t-i+1}-x_0^{t-i}\|^2.
\end{align}
Then for any given $t$, the difference $L(\{x^{t+1}_k\}, x_0^{t+1}; y^{t+1})-L(\{x^{1}_k\}, x_0^{1}; y^{1})$ is obtained by summing \eqref{eq:L_difference2} {over all} iterations. Specifically, we obtain
\begin{align}
&L(\{x^{t+1}_k\}, x_0^{t+1}; y^{t+1})-L(\{x^{1}_k\}, x_0^{1}; y^{1})\nonumber\\
&\le -\sum_{i=1}^{t}\sum_{k=1}^{K}\left({\frac{\rho_k-7L_k}{2}}-\left(\frac{4L_k}{\rho^2_k}+\frac{1}{\rho_k}\right)2L_k^2\right)\|x^{i+1}_k-x^i_k\|^2
\nonumber\\
&\quad\quad\quad \quad -\sum_{i=1}^{t}\sum_{k=1}^{K}\left(\frac{\rho_k}{2}-T^2\left(\frac{4L_k}{\rho^2_k}+\frac{1}{\rho_k}\right)8L_k^2\right)\|x_0^{i+1}-x_0^i\|^2\nonumber\\
&\:= -\sum_{i=1}^{t}\sum_{k=1}^{K}\alpha_k \|x^{i+1}_k-x^i_k\|^2
-\sum_{i=1}^{t}\sum_{k=1}^{K}\beta_k \|x_0^{i+1}-x_0^i\|^2\nonumber.
\end{align}
This completes the proof.
\end{proof}

We conclude that to make the rhs of \eqref{eq:L_descent2} negative at each iteration, it is sufficient to require that  $\alpha_k>0$ and $\beta_k>0$ for all $k$, or more specifically:
\begin{align}\label{eq:rho_condition2}
\begin{split}
&{\frac{\rho_k-7L_k}{2}}-\left(\frac{4L_k}{\rho^2_k}+\frac{1}{\rho_k}\right)2L_k^2>0,\; k=1,\cdots, K,\\
&\frac{\rho_k}{2}-T^2\left(\frac{4L_k}{\rho^2_k}+\frac{1}{\rho_k}\right)8L_k^2>0,\; k=1,\cdots, K.
\end{split}
\end{align}
Note that one can always find a set of $\rho_k$'s large enough such that the above condition is satisfied. 

Next we show that $L(\{x^{t}_k\}, x_0^{t}; y^{t})$ is convergent.
\begin{lemma}\label{lemma:L_bounded2}
{Suppose Assumption A1, A3 and Assumption B are} satisfied. Then Algorithm 3 with period-$T$ essentially cyclic update rule generates a sequence of augmented Lagrangian, {whose limit exists and is bounded below by $\underline{f}$.}
\end{lemma}
\begin{proof}
Observe that the augmented Lagrangian can be expressed as
\begin{align}\label{eq:L_lower_bound2}
&L(\{x^{t+1}_k\}, x_0^{t+1}; y^{t+1})\nonumber\\
&=h(x_0^{t+1})+\sum_{k=1}^{K}\left(g_k(x^{t+1}_k)+\langle y^{t+1}_k, x^{t+1}_k-x_0^{t+1}\rangle+\frac{\rho_k}{2}\|x^{t+1}_k-x_0^{t+1}\|^2\right)\nonumber\\
&\stackrel{\rm (a)}=h(x_0^{t+1})+\sum_{k=1}^{K}\left(g_k(x^{t+1}_k)+\langle \nabla g_k(x_0^{t+1})+L_k(x^{t+1}_k-x_0^{t+1}), x_0^{t+1}-x_k^{t+1}\rangle+\frac{\rho_k}{2}\|x^{t+1}_k-x_0^{t+1}\|^2\right)\nonumber\\
&=h(x_0^{t+1})+\sum_{k=1}^{K}\left(g_k(x^{t+1}_k)+\langle \nabla g_k(x_0^{t+1}), x_0^{t+1}-x_k^{t+1}\rangle+\frac{\rho_k-2L_k}{2}\|x^{t+1}_k-x_0^{t+1}\|^2\right)\nonumber\\
&\stackrel{\rm (b)}\ge h(x_0^{t+1})+\sum_{k=1}^{K}\left(g_k(x_0^{t+1})+\frac{\rho_k-5L_k}{2}\|x^{t+1}_k-x_0^{t+1}\|^2\right)\nonumber\\
&= f(x_0^{t+1})+\sum_{k=1}^{K}\frac{\rho_k-5L_k}{2}\|x^{t+1}_k-x_0^{t+1}\|^2
\end{align}
where $\rm (a)$ is from \eqref{eq:y_expression2}; $\rm (b)$ is due to the following inequalities
\begin{align}
g_k(x_0^{t+1})&\le g_k(x^{t+1}_k)+\langle\nabla g_k(x^{t+1}_k), x_0^{t+1}-x^{t+1}_k\rangle+\frac{L_k}{2}\|x^{t+1}_k-x_0^{t+1}\|^2\nonumber\\
&= g_k(x^{t+1}_k)+\langle\nabla g_k(x^{t+1}_k)-\nabla g_k(x_0^{t+1}), x_0^{t+1}-x^{t+1}_k\rangle\nonumber\\
&\quad \quad +\langle\nabla g_k(x_0^{t+1}), x_0^{t+1}-x^{t+1}_k\rangle+\frac{L_k}{2}\|x^{t+1}_k-x_0^{t+1}\|^2\nonumber\\
&\le g_k(x^{t+1}_k)+\langle\nabla g_k(x_0^{t+1}), x_0^{t+1}-x^{t+1}_k\rangle+\frac{3L_k}{2}\|x^{t+1}_k-x_0^{t+1}\|^2.\nonumber
\end{align}
Clearly, combining the inequality \eqref{eq:L_lower_bound2} with Assumptions B and A3 yields that $L(\{x^{t+1}_k\}, x_0^{t+1}; y^{t+1})$ is lower bounded. It follows from Lemma \ref{lemma:L_difference2} that whenever the penalty parameter $\rho_k$'s are chosen sufficiently large (as per Assumption B), $L(\{x^{t+1}_k\}, x_0^{t+1}; y^{t+1})$ will monotonically decrease and is convergent. This completes the proof.
\end{proof}

Using Lemmas \ref{lemma:y2}--\ref{lemma:L_bounded2}, we arrive at the following convergence result. The proof is similar to Theorem \ref{thm:convergence1}, and is thus omitted.
\begin{theorem}\label{thm:convergence2}
{Suppose that Assumptions A1, A3 and B hold. Then the following is true for Algorithm 3.
\begin{enumerate}
\item We have $\lim_{t\to\infty}\|x_0^{t+1}-x^{t+1}_k\|=0$, $k=1,\cdots, K$.
\item Let $(\{x^*_k\}, x_0^*, y^*)$ denote any limit point of the sequence
$\{\{x^{t+1}_k\}, x_0^{t+1}, y^{t+1}\}$ generated by Algorithm 3 with period-$T$ essentially cyclic block update rule.  Then $(\{x^*_k\}, x_0^*, y^*)$ is a stationary solution of problem \eqref{eq:consensus:admm}.
\item If $X$ is a compact set, then Algorithm 3 with period-T essentially cyclic block update rule converges to the set of stationary solutions of problem \eqref{eq:consensus:admm}. That is, the following is true
    \begin{align}
    \lim_{t\to\infty}\mbox{\rm dist}\left( (\{x^t_k\}, x_0^t, y^t); Z^* \right)=0.
    \end{align}
    where $Z^*$ is the set of primal-dual stationary solutions of problem \eqref{eq:consensus:admm}.
\end{enumerate}
}
\end{theorem}

\section{The Nonconvex Sharing Problem}\label{sec:sharing}
Consider the following well-known sharing problem (see, e.g., \cite[Section 7.3]{BoydADMM} for motivation)
\begin{align}
\begin{split}\label{eq:sharing}
\min&\quad f(x_1,\cdots,x_K):=\sum_{k=1}^{K}g_k(x_k)+\ell\left(\sum_{k=1}^{K}A_k x_k\right)\\
\st &\quad x_k\in X_k,\; k=1,\cdots, K,
\end{split}
\end{align}
{where $x_k\in\mathbb{R}^{N_k}$ is the variable associated with a given agent $k$, and $A_k\in\mathbb{R}^{M\times N_k}$ is some data matrix}. The variables are coupled through the function $\ell(\cdot)$.

To facilitate distributed computation, this problem can be equivalently formulated into a linearly constrained problem by introducing an additional variable $x_0\in \mathbb{R}^{M}$:
\begin{align}
\begin{split}\label{eq:sharing:admm}
\min&\quad \sum_{k=1}^{K}g_k(x_k)+\ell\left(x_0\right)\\
\st &\quad \sum_{k=1}^{K}A_k x_k=x_0,\quad  x_k\in X_k,\; k=1,\cdots, K.
\end{split}
\end{align}
The augmented Lagrangian for this problem is given by
\begin{align}\label{eq:lagrangian:sharing}
L(\{x_k\}, x_0; y)=\sum_{k=1}^{K}g_k(x_k)+\ell(x_0)+\bigg\langle x_0-\sum_{k=1}^{K}A_k x_k, y\bigg\rangle+\frac{\rho}{2}\bigg\|x_0-\sum_{k=1}^{K}A_k x_k\bigg\|^2.
\end{align}

Note that we have chosen a special reformulation in \eqref{eq:sharing:admm}: a {\it single} variable $x_0$ is introduced which leads to a problem with a {\it single} linear constraint.  Applying the classical ADMM to this reformulation leads to a {\it multi-block} ADMM algorithm in which $K+1$ block variables $(\{x_k\}_{k=1}^{K}, x_0)$ are updated sequentially. As mentioned in {the} introduction, even in the case where the objective is convex, it is not known whether the multi-block ADMM converges in this case. Variants of the multi-block ADMM has been proposed in the literature to solve this type of multi-block problems; see recent developments in \cite{HongLuo2012ADMM, he:alternating12, chen13, hong13BSUMM, Wang13} and the references therein.

{In this section, we show that the classical ADMM, together with several of its extensions using different block selection rules, converge even when the objective function is nonconvex.} The main assumptions for convergence are that the penalty parameter $\rho$ is large enough, and that the coupling function $\ell(x_0)$ should be smooth (more detailed conditions will be given shortly).
Similarly as in the previous sections, we consider a generalized version of  ADMM with two types of block update rules: the period-$T$ essentially cyclic rule and the randomized rule. The detailed algorithm is given in the table below.

\begin{center}
\fbox{
\begin{minipage}{4.9in}
\smallskip
\centerline{\bf Algorithm 4. The Flexible ADMM for Problem \eqref{eq:sharing:admm}}
\smallskip
{Let $\cC^{1}=\{0,\cdots, K\}$, $t=0, 1, \cdots$.

At each iteration $t+1$, do:\\

{\bf If} $t+1\ge 2$, pick an index set $\cC^{t+1}\subseteq\{0,\cdots,K\}$}.\\

{\bf For} $k=1,\cdots, K$

{\bf If} \quad $k\in\cC^{t+1}$, then agent $k$ updates $x_k$ by:
\begin{align}\label{eq:x_k_update6}
x_k^{t+1}&={\rm arg}\!\min_{x_k\in X_k} g_k(x_k)-\langle y^t, A_k x_k\rangle+\frac{\rho}{2}\bigg\|x_0^t-\sum_{j<k}A_j x_j^{t+1}-\sum_{j>k}A_j x_j^{t}-A_k x_k\bigg\|^2
\end{align}
{\bf Else} $x^{t+1}_k=x_{k}^{t}$.

{\bf If} $0\in\cC^{t+1}$, update the variable $x_0$ by:
\begin{align}\label{eq:x_update6}
x_0^{t+1}=\arg\min_{x} \ell(x_0)+\langle y^{t}, x_0\rangle+\frac{\rho}{2}\left\|x_0-\sum_{k=1}^{K}A_k x^{t+1}_k\right\|^2.
\end{align}
\; \; Update the dual variable:
\begin{align}
y^{t+1}=y^t+\rho\left(x_0^{t+1}-\sum_{k=1}^{K}A_k x_k^{t+1}\right).
\end{align}
{\bf Else} $x_0^{t+1}=x_0^{t}$, $y^{t+1}=y^t$.
\end{minipage}
}
\end{center}

The analysis of Algorithm 4 follows similar argument as that of Algorithm 3. Therefore we will only provide an outline for it.

First, we make the following assumptions in this section.

\pn {\bf Assumption C.}
\begin{itemize}
\item[C1.] There exists a positive constant $L>0$ such that $$\|\nabla \ell(x)-\nabla \ell(z)\|\le L \|x-z\|, \; \forall~x,z.$$    Moreover, $X_k$'s are closed convex sets; each $A_k$ is full column rank so that {$\lambda_{\rm min}(A_k^{T} A_k)>0$, where $\lambda_{\rm min}$ denotes the minimum eigenvalue of a matrix}.
\item[C2.] The penalty parameter $\rho$ is chosen large enough such that:
\begin{enumerate}
\item [(1)] {Each $x_k$ subproblem \eqref{eq:x_k_update6}} as well as the $x_0$ subproblem \eqref{eq:x_update6} is strongly convex, with modulus $\{\gamma_k(\rho)\}_{k=1}^{K}$ and $\gamma(\rho)$, respectively.
\item[(2)] $\rho\gamma(\rho)> 2 L^2$, and that $\rho\ge L$.
\end{enumerate}
\item[C3.] $f(x_1,\cdots, x_K)$ is lower bounded over $\prod_{k=1}^KX_k$.
\item[C4.] $g_k$ is either smooth nonconvex or convex (possibly nonsmooth). For the former case, there exists $L_k>0$ such that $\|{\nabla} g_k(x_k)-{\nabla}g_k(z_k)\|\le L_k\|x_k-z_k\|$, $\forall~x_k, z_k\in X_k$.
\end{itemize}

Note that compared with Assumptions A and B, in this case we no longer require that each $g_k$ to be smooth. Define an index set $\cK\subseteq \{1,\cdots,K\}$, such that $g_k$ is convex if $k\in\cK$, and nonconvex smooth otherwise. Further, the requirement that $A_k$ is full column rank is needed to make the $x_k$ subproblem \eqref{eq:x_k_update6} strongly convex. 

Our convergence analysis consists of a series of lemmas whose proofs, for the most part, are omitted since they are similar to that of Lemma \ref{lemma:y1}--Lemma \ref{lemma:L_bounded1}.
\begin{lemma}\label{lemma:y3}
Suppose Assumption C is satisfied. Then for Algorithm 4 with either essentially cyclic rule or the randomized rule, the following is true
\begin{align}
&\nabla\ell(x_0^{t+1})=-y^{t+1}, \; \mbox{if}~~0\in\cC^{t+1},\quad L^2\|x_0^{t+1}-x_0^t\|^2\ge \|y^{t+1}-y^t\|^2, \nonumber\\
&L^2\|x_0^{t+1}-x_0^{t(k)}\|^2\ge \|y^{t+1}-y^{t(k)}\|^2, \; {L^2\|\hx_0^{t+1}-x_0^{t}\|^2\ge \|\hy^{t+1}-y^{t}\|^2}.\nonumber
\end{align}
\end{lemma}

\begin{lemma}\label{lemma:L_difference3}
Suppose Assumption C is satisfied. Then for Algorithm 4 with either essentially cyclic rule or the randomized rule, the following is true
\begin{align}
&L(\{x^{t+1}_k\}, x_0^{t+1}; y^{t+1})-L(\{x^{t}_k\}, x_0^{t}; y^{t})\nonumber\\
&\le \sum_{k\ne 0, k\in\cC^{t+1}}-\frac{\gamma_k(\rho)}{2}\|x_k^{t+1}-x_{k}^{t}\|^2-\left(\frac{\gamma(\rho)}{2}-\frac{L^2}{\rho}\right)\|x_0^{t+1}-x_0^t\|^2.
\end{align}
\end{lemma}

\begin{lemma}\label{lemma:L_bounded3}
Assume the same set up as in Lemma \ref{lemma:L_difference3}. {Then the following limit exists and is bounded from below
\begin{align}
\lim_{t\to\infty} L(\{x^{t+1}_k\}, x_0^{t+1}; y^{t+1}).
\end{align}}
\end{lemma}
\begin{proof}
We have the following series of inequalities
\begin{align}
&L(\{x_k^{r+1}\}, x_0^{r+1}; y^{r+1})\nonumber\\
&=\sum_{k=1}^{K}g_k(x^{t+1}_k)+\ell(x_0^{t+1})+\bigg\langle x_0^{t+1}-\sum_{k=1}^{K}A_k x^{t+1}_k, y^{t+1}\bigg\rangle+\frac{\rho}{2}\bigg\|x_0^{t+1}-\sum_{k=1}^{K}A_k x^{t+1}_k\bigg\|^2\nonumber\\
&= \sum_{k=1}^{K}g_k(x^{t+1}_k)+\ell(x_0^{t+1})+\bigg\langle \sum_{k=1}^{K}A_k x^{t+1}_k-x_0^{t+1}, \nabla \ell(x_0^{t+1})\bigg\rangle+\frac{\rho}{2}\bigg\|x_0^{t+1}-\sum_{k=1}^{K}A_k x^{t+1}_k\bigg\|^2\nonumber\\
&\ge \sum_{k=1}^{K}g_k(x^{t+1}_k)+\ell\left(\sum_{k=1}^{K}A_k x^{t+1}_k\right)+\frac{\rho-L}{2}\bigg\|x_0^{t+1}-\sum_{k=1}^{K}A_k x^{t+1}_k\bigg\|^2.\nonumber
\end{align}
The last inequality comes from the fact that
\begin{align}
\ell\left(\sum_{k=1}^{K}A_k x^{t+1}_k\right)&\le \ell(x_0^{t+1})+\bigg\langle \sum_{k=1}^{K}A_k x^{t+1}_k-x_0^{t+1}, \nabla \ell(x_0^{t+1})\bigg\rangle+\frac{L}{2}\left\|x_0^{t+1}-\sum_{k=1}^{K}A_k x^{t+1}_k\right\|^2\nonumber.
\end{align}
Using assumptions C2.-- C3. leads to the desired result.
\end{proof}

{ We note that the above result holds true deterministically even if the randomized scheme is used. The reason is that at each iteration regardless of whether $0\in\cC^{t+1}$, we have $y^{t+1}=-\nabla \ell(x^{t+1})$ because these two variables are always updated at the same iteration. The rest of the proof is not dependent on the algorithm. }

We have the following main result for the nonconvex consensus problem.
\begin{theorem}\label{thm:convergence3}
{Suppose that Assumption C holds. Then the following is true for Algorithm 4, either deterministically for the essentially cyclic update rule or almost surely for the randomized update rule.
\begin{enumerate}
\item We have $\lim_{t\to\infty}\|\sum_{k}A_kx^{t+1}_k-x_0^{t+1}\|=0$, $k=1,\cdots, K$.
\item Let $(\{x^*_k\}, x_0^*, y^*)$ denote any limit point of the sequence
$\{\{x^{t+1}_k\}, x_0^{t+1}, y^{t+1}\}$ generated by Algorithm 4. Then $(\{x^*_k\}, x_0^*, y^*)$ is a stationary solution of problem \eqref{eq:sharing:admm} in the sense that 
\begin{align}
&x_k^*\in \arg\min_{x_k\in X_k} \; g_k(x_k)+\langle y^*, -A_k x_k\rangle,\; k\in\cK,\nonumber\\
& \left \langle x_k-x^*_k, \nabla g_k(x^*_k)- A^T_ky^*\right\rangle\ge 0,\; \forall~x_k\in X_k, \; k\notin\cK,\nonumber\\
&\nabla \ell(x_0^*)+y^* = 0, \nonumber\\
&\sum_{k=1}^{K}A_kx_k^*=x_0^*. \nonumber
\end{align}
\item  {If $X_k$ is a compact set for all $k$, then Algorithm 4 converges to the set of stationary solutions of problem \eqref{eq:sharing:admm}, i.e.,
        \begin{align}
    \lim_{t\to\infty}\mbox{\rm dist}\left( (\{x^t_k\}, x_0^t, y^t); Z^* \right)=0,
    \end{align}
    where $Z^*$ is the set of primal-dual stationary solution for problem \eqref{eq:sharing:admm}.}
\end{enumerate}
}
\end{theorem}

The following corollary specializes the previous convergence result to the case where all $g_k$'s as well as $\ell$ are convex (but not necessarily strongly convex). We emphasize that this is still a nontrivial result, since unlike \cite{HongLuo2012ADMM, hong13BSUMM, han12admm, lin_ma_zhang2014}, we do not require the dual stepsize to be small or the $g_k$'s and $\ell$ to be strongly convex. Therefore it is not known whether the classical ADMM converges for the multi-block problem \eqref{eq:sharing:admm}, even for the convex case.
\begin{corollary}
Suppose that Assumptions C1 and C3 hold, and that $g_k$ and $\ell$ are convex. Further, suppose that Assumption C2 is weakened with the following assumption
\begin{enumerate}
\item[C2'] The penalty parameter $\rho$ is chosen large enough such that $\rho > \sqrt{2} L$.
\end{enumerate}
Then the flexible ADMM algorithm (i.e., Algorithm 4), converges to the set of primal dual optimal solution $(\{x^*_k\}, x^*, y^*)$ of problem \eqref{eq:consensus:admm}, either deterministically for the essentially cyclic update rule or almost surely for the randomized update rule.
\end{corollary}


Similar to the consensus problem, one can extend Algorithm 4 to its proximal version. Here the benefit offered by the proximal-type algorithms is twofold: {\it i}) one can remove the strong convexity requirement posed in Assumption C2-(1) ; {\it ii)}  one can allow inexact and simple update for each block variable. However, the analysis is a bit more involved, as the penalty parameter $\rho$ as well as the proximal coefficient for each subproblem needs to be carefully bounded. Due to the fact that the analysis follows almost identical steps as those in Section \ref{sub:consensus_proximal}, we will not present them here.



\section{Extensions}\label{sec:main_approach}
In this paper, we analyze the behavior of the ADMM method in the absence of convexity. We show that when the penalty parameter is chosen sufficiently large, the ADMM and several of its variants converge to the set of stationary solutions for certain consensus and sharing problems.

{Our analysis is based on using the augmented Lagrangian as a potential function to guide the iterate convergence. This approach may be extended to other nonconvex problems. In particular, if the following set of sufficient conditions (see Assumption D below) are satisfied, then the convergence of the ADMM is guaranteed for the nonconvex problem \eqref{problem:Original}. {It is important to note that in practice these conditions should be verified case by case for different applications, just like what we have done for the consensus and sharing problems. }

\pn {\bf Assumption D}
\begin{itemize}
\item[D1.] The iterations are well defined, meaning the function $L(x^t; y^t)$ is uniformly lower bounded for all $t$.
\item[D2.] There exists a constant $\sigma>0$ such that $\|y^{t+1}-y^{t}\|^2\le \sigma\|x^{t+1}-x^{t}\|^2$, for all $t$.
\item[D3.] $g_k(\cdot)$ is either smooth nonconvex or nonsmooth convex. The coupling function $\ell(\cdot)$ is smooth with Lipschitz continuous gradient $L$. Moreover, $\ell(\cdot)$ is convex with respect to each block variable $x_k$, but is not necessarily jointly convex with $x$. $X_k$ is a closed convex set. Problem \eqref{problem:Original} is feasible, that is, $\{x\mid Ax=q\}\bigcap_{k=1}^{K} {\rm relint} X_k\ne \emptyset$.
\item[D4.] The penalty parameter $\rho$ is chosen large enough such that each subproblem is strongly convex with modulus $\gamma_k(\rho)$, which is a nondecreasing function of $\rho$. Further, $\rho\gamma_k(\rho)>2\sigma$ for all $k$.
\end{itemize}

Following a similar argument leading to Theorem \ref{thm:convergence1}, we can show that as long as Assumption D is satisfied, then the primal feasibility gap $\|q-\sum_{k=1}^{K}A_k x^{t+1}_k\|$ goes to zero in the limit, and that every limit point of the sequence $\{\{x^{t+1}_k\}, x_0^{t+1}, y^{t+1}\}$ is a stationary solution of problem \eqref{problem:Original}. {A few remarks on Assumption D are in order:
\begin{enumerate}
\item Assumption D1 is necessary for showing convergence. Without D1, even if one is able to show that the augmented Lagrangian is decreasing, one cannot claim the convergence to stationary solutions. The reason is that the augmented Lagrangian may go to $-\infty$ \footnote{In fact, it is very easy to modify the algorithm so that the  augmented Lagrangian reduces at each iteration -- just change the ``+" in the dual update \eqref{eq:y_update} to ``-". However, it is obvious that by doing this the dual variables will become unbounded, and the primal feasibility will never be satisfied. }, therefore there is no way to guarantee that the successive difference of the iterates goes to $0$, or the primal feasibility is satisfied in the limit.

\item The main drawback of Assumption D is that it is made on the iterates rather than on the problem. For different linearly constrained optimization problems, one still needs to verify that these conditions are indeed valid, as we have done for the consensus and the sharing problem considered in this paper.

\end{enumerate}
}

{Here we mention one more  family of problems for which Assumption D can be verified. Consider
\begin{align}
\min&\quad f(x_1) + g(x_2)\nonumber\\
\st&\quad Bx_1 + Ax_2=c, \; x_1\in X,
\end{align}
where $f(\cdot)$ is a convex possibly nonsmooth function; $g(\cdot)$ is a possibly nonconvex function, and has Lipschitzian gradient with modulus $L_g$; $X\subseteq R^{N}$; $A$ is an invertible matrix; $g(\cdot)$ and $f(\cdot)$ are lower bounded over the set $X$. Consider the following ADMM method, where the iterate generated at iteration $t+1$ is given by
\begin{align}
x^{t+1}_1 & = \arg\min_{x_1\in X}\; f(x_1) +\langle B x_1 + A x^t_2-c, y^t\rangle+\frac{\rho}{2}\|B x_1 + A x^t_2-c\|^2 \nonumber\\
x^{t+1}_2 & = \arg\min \; g(x_2) +\langle B x^{t+1}_1 + A x_2-c, y^t\rangle+\frac{\rho}{2}\|B x^{t+1}_1 + A x_2-c\|^2 \nonumber\\
y^{t+1} &= y^{t}+\rho\left(B x^{t+1}_1 + A x^{t+1}_2-c\right)\nonumber.
\end{align}
{By using steps in Lemma 2.1-Lemma 2.3, one can verify that if $\rho>{L_g}/{\lambda_{\min}(AA^T)}$, then Assumptions D1 holds true. By having $\rho$ large enough and by using the invertibility of $A$, we can make the $x_2$ subproblem strongly convex, then Assumption D4 holds true.  Other assumptions can be verified along similar lines. Note that in this case the convergence can be obtained with a slightly weaker condition in which the $x_1$ subproblem is convex but not necessarily strongly convex.}

}

\newpage
\bibliographystyle{unsrt}

\bibliography{ref,biblio}

\begin{thebibliography}{10}

\bibitem{Hong2015icassp}
M.~Hong, Z.-Q. Luo, and M.~Razaviyayn.
\newblock On the convergence of alternating direction method of mulitpliers for
  a family of nonconvex problems.
\newblock In {\em ICASSP 2015}, 2015.

\bibitem{ADMMGlowinskiMorroco}
R.~Glowinski and A.~Marroco.
\newblock Sur l'approximation, par elements finis d'ordre un,et la resolution,
  par penalisation-dualite, d'une classe de problemes de dirichlet non
  lineares.
\newblock {\em Revue Franqaise d'Automatique, Informatique et Recherche
  Opirationelle}, 9:41--76, 1975.

\bibitem{ADMMGabbayMercier}
D.~Gabay and B.~Mercier.
\newblock A dual algorithm for the solution of nonlinear variational problems
  via finite element approximation.
\newblock {\em Computers $\&$ Mathematics with Applications}, 2:17--40, 1976.

\bibitem{Eckstein89}
J.~Eckstein.
\newblock Splitting methods for monotone operators with applications to
  parallel optimization.
\newblock 1989.
\newblock Ph.D Thesis, Operations Research Center, MIT.

\bibitem{EcksteinBertsekas1992}
J.~Eckstein and D.~P. Bertsekas.
\newblock On the {D}ouglas-{R}achford splitting method and the proximal point
  algorithm for maximal monotone operators.
\newblock {\em Mathematical Programming}, 55(1):293--318, 1992.

\bibitem{Glow84}
R.~Glowinski.
\newblock {\em Numerical methods for nonlinear variational problems}.
\newblock Springer-Verlag, New York, 1984.

\bibitem{bertsekas97}
D.~P. Bertsekas and J.~N. Tsitsiklis.
\newblock {\em Parallel and Distributed Computation: Numerical Methods, 2nd
  ed}.
\newblock Athena Scientific, Belmont, MA, 1997.

\bibitem{BoydADMM}
S.~Boyd, N.~Parikh, E.~Chu, B.~Peleato, and J.~Eckstein.
\newblock Distributed optimization and statistical learning via the alternating
  direction method of multipliers.
\newblock {\em Foundations and Trends in Machine Learning}, 3, 2011.

\bibitem{Yin:2008:BIA:1658318.1658320}
W.~Yin, S.~Osher, D.~Goldfarb, and J.~Darbon.
\newblock Bregman iterative algorithms for l1-minimization with applications to
  compressed sensing.
\newblock {\em SIAM Journal on Imgaging Science}, 1(1):143--168, March 2008.

\bibitem{Yang09TV}
J.~Yang, Y.~Zhang, and W.~Yin.
\newblock An efficient {TVL1} algorithm for deblurring multichannel images
  corrupted by impulsive noise.
\newblock {\em SIAM Journal on Scientific Computing}, 31(4):2842--2865, 2009.

\bibitem{zhang11primaldual}
X.~Zhang, M.~Burger, and S.~Osher.
\newblock A unified primal-dual algorithm framework based on {B}regman
  iteration.
\newblock {\em Journal of Scientific Computing}, 46(1):20--46, 2011.

\bibitem{Scheinberg10inverse}
K.~Scheinberg, S.~Ma, and D.~Goldfarb.
\newblock Sparse inverse covariance selection via alternating linearization
  methods.
\newblock In {\em Advanced in Neural Information Processing Systems (NIPS)},
  2010.

\bibitem{Schizas08}
I.~Schizas, A.~Ribeiro, and G.~Giannakis.
\newblock Consensus in ad hoc wsns with noisy links - part i: Distributed
  estimation of deterministic signals.
\newblock {\em IEEE Transactions on Signal Processing}, 56(1):350 -- 364, 2008.

\bibitem{Feng14}
C.~Feng, H.~Xu, and B.~Li.
\newblock An alternating direction method approach to cloud traffic management.
\newblock {\em submitted to IEEE/ACM Trans. Networking}, 2014.

\bibitem{liao14sdn}
W.-C. Liao, M.~Hong, Hamid Farmanbar, Xu~Li, Z.-Q. Luo, and Hang Zhang.
\newblock Min flow rate maximization for software defined radio access
  networks.
\newblock {\em IEEE Journal on Selected Areas in Communication},
  32(6):1282--1294, 2014.

\bibitem{bertsekas99}
D.~P. Bertsekas.
\newblock {\em Nonlinear Programming, 2nd ed}.
\newblock Athena Scientific, Belmont, MA, 1999.

\bibitem{boyd04}
S.~Boyd and L.~Vandenberghe.
\newblock {\em Convex Optimization}.
\newblock Cambridge University Press, 2004.

\bibitem{Nedic09}
A.~Nedic and A.~Ozdaglar.
\newblock Cooperative distributed multi-agent optimization.
\newblock In {\em Convex Optimization in Signal Processing and Communications}.
  Cambridge University Press, 2009.

\bibitem{bertsekas82}
D.~P. Bertsekas.
\newblock {\em Constrained Optimization and Lagrange Multiplier Method}.
\newblock Academic Press, 1982.

\bibitem{HeYuan2012}
B.~He and X.~Yuan.
\newblock On the {O}(1/n) convergence rate of the {D}ouglas-{R}achford
  alternating direction method.
\newblock {\em SIAM Journal on Numerical Analysis}, 50(2):700--709, 2012.

\bibitem{Monteiro13}
R.~Monteiro and B.~Svaiter.
\newblock Iteration-complexity of block-decomposition algorithms and the
  alternating direction method of multipliers.
\newblock {\em SIAM Journal on Optimization}, 23(1):475--507, 2013.

\bibitem{Davis14}
D.~Davis and W.~Yin.
\newblock Convergence rate analysis of several splitting schemes.
\newblock 2014.
\newblock UCLA CAM Report 14-51.

\bibitem{goldstein12}
T.~Goldstein, B.~O'Donoghue, S.~Setzer, and R.~Baraniuk.
\newblock Fast alternating direction optimization methods.
\newblock {\em SIAM Journal on Imaging Sciences}, 7(3):1588--1623, 2014.

\bibitem{goldfarb12}
D.~Goldfarb, S.~Ma, and K.~Scheinberg.
\newblock Fast alternating linearization methods for minimizing the sum of two
  convex functions.
\newblock {\em Mathematical Programming}, 141(1-2):349--382, 2012.

\bibitem{ADMMlinearYin}
W.~Deng and W.~Yin.
\newblock On the global linear convergence of alternating direction methods.
\newblock 2012.
\newblock preprint.

\bibitem{chen13}
C.~Chen, B.~He, X.~Yuan, and Y.~Ye.
\newblock The direct extension of {ADMM} for multi-block convex minimization
  problems is not necessarily convergent.
\newblock 2013.
\newblock Mathematical Programming, to appear.

\bibitem{HongLuo2012ADMM}
M.~Hong and Z.-Q. Luo.
\newblock On the linear convergence of the alternating direction method of
  multipliers.
\newblock {\em arXiv preprint arXiv:1208.3922}, 2012.

\bibitem{he:alternating12}
B.~He, M.~Tao, and X.~Yuan.
\newblock Alternating direction method with {G}aussian back substitution for
  separable convex programming.
\newblock {\em SIAM Journal on Optimization}, 22:313--340, 2012.

\bibitem{hong13BSUMM}
M.~Hong, T.-H. Chang, X.~Wang, M.~Razaviyayn, S.~Ma, and Z.-Q. Luo.
\newblock A block successive upper bound minimization method of multipliers for
  linearly constrained convex optimization.
\newblock 2013.
\newblock Preprint, available online arXiv:1401.7079.

\bibitem{Wang13}
X.~Wang, M.~Hong, S.~Ma, and Z.-Q. Luo.
\newblock Solving multiple-block separable convex minimization problems using
  two-block alternating direction method of multipliers.
\newblock {\em Pacific Journal on Optimization}, 11(4):645--667, 2015.

\bibitem{han12admm}
D.~Han and X.~Yuan.
\newblock A note on the alternating direction method of multipliers.
\newblock {\em Journal of Optimization Theory and Applications},
  155(1):227--238, 2012.

\bibitem{Deng14Parallel}
W.~Deng, M.~Lai, Z.~Peng, and W.~Yin.
\newblock Parallel multi-block {ADMM} with o(1/k) convergence.
\newblock Preprint, available online at arXiv: 1312.3040., 2014.

\bibitem{HeXuYuan2013}
B.~He, H.~Xu, and X.~Yuan.
\newblock On the proximal jacobian decomposition of alm for multipleblock
  separable convex minimization problems and its relationship to {ADMM}.
\newblock 2013.
\newblock Preprint, available on Optimization-Online.

\bibitem{lin_ma_zhang2014}
T.~Lin, S.~Ma, and S.~Zhang.
\newblock On the global linear convergence of the admm with multi-block
  variables.
\newblock 2014.
\newblock Preprint.

\bibitem{hong13bcdmm_icassp}
M.~Hong, T.-H. Chang, X.~Wang, M.~Razaviyayn, S.~Ma, and Z.-Q. Luo.
\newblock A block coordinate descent method of multipliers: Convergence
  analysis and applications.
\newblock In {\em International Conference on Acoustics, Speech and Signal
  Processing}, 2014.

\bibitem{Gao15:nonseparate}
X.~Gao and S.~Zhang.
\newblock First-order algorithms for convex optimization with nonseparate
  objective and coupled constraints.
\newblock 2015.
\newblock Preprint.

\bibitem{zhang10ADMM_NMF}
Y.~Zhang.
\newblock An alternating direction algorithm for nonnegative matrix
  factorization.
\newblock 2010.
\newblock Preprint.

\bibitem{sun14}
D.~L. Sun and C.~Fevotte.
\newblock Alternating direction method of multipliers for non-negative matrix
  factorization with the beta-divergence.
\newblock In {\em the Proceedings of IEEE International Conference on
  Acoustics, Speech, and Signal Processing (ICASSP)}, 2014.

\bibitem{wen12}
Z.~Wen, C.~Yang, X.~Liu, and S.~Marchesini.
\newblock Alternating direction methods for classical and ptychographic phase
  retrieval.
\newblock {\em Inverse Problems}, 28(11):1--18, 2012.

\bibitem{zhang14}
R.~Zhang and J.~T. Kwok.
\newblock Asynchronous distributed admm for consensus optimization.
\newblock In {\em Proceedings of the 31st International Conference on Machine
  Learning}, 2014.

\bibitem{Forero11}
P.A. Forero, A.~Cano, and G.B. Giannakis.
\newblock Distributed clustering using wireless sensor networks.
\newblock {\em IEEE Journal of Selected Topics in Signal Processing},
  5(4):707--724, Aug 2011.

\bibitem{Ames13LDA}
B.~Ames and M.~Hong.
\newblock Alternating directions method of multipliers for l1-penalized zero
  variance discriminant analysis and principal component analysis.
\newblock 2014.
\newblock Preprint.

\bibitem{jiang13ADMM}
B.~Jiang, S.~Ma, and S.~Zhang.
\newblock Alternating direction method of multipliers for real and complex
  polynomial optimization models.
\newblock 2013.
\newblock Preprint.

\bibitem{Liavas14}
A.~P. Liavas and N.~D. Sidiropoulos.
\newblock Parallel algorithms for constrained tensor factorization via the
  alternating direction method of multipliers.
\newblock 2014.
\newblock Preprint, available at arXiv:1409.2383v1.

\bibitem{Shen:2014}
Y.~Shen, Z.~Wen, and Y.~Zhang.
\newblock Augmented lagrangian alternating direction method for matrix
  separation based on low-rank factorization.
\newblock {\em Optimization Methods Software}, 29(2):239--263, March 2014.

\bibitem{xu11admm_matrix}
Y.~Xu, W.~Yin, Z.~Wen, and Y.~Zhang.
\newblock An alternating direction algorithm for matrix completion with
  nonnegative factors.
\newblock {\em Journal of Frontiers of Mathematics in China, Special Issues on
  Computational Mathematics}, pages 365--384, 2011.

\bibitem{Wen13Risk}
Z.~Wen, X.~Peng, X.~Liu, X.~Bai, and X.~Sun.
\newblock Asset allocation under the basel accord risk measures.
\newblock 2013.
\newblock Preprint.

\bibitem{lin13}
F.~Lin, M.~Fardad, and M.~R. Jovanovic.
\newblock Design of optimal sparse feedback gains via the alternating direction
  method of multipliers.
\newblock {\em IEEE Transactions on Automatic Control}, 58(9):2426--2431, Sept
  2013.

\bibitem{zhang10ADMM_QP}
Y.~Zhang.
\newblock Convergence of a class of stationary iterative methods for saddle
  point problems.
\newblock 2010.
\newblock Preprint.

\bibitem{Razaviyayn12SUM}
M.~Razaviyayn, M.~Hong, and Z.-Q. Luo.
\newblock A unified convergence analysis of block successive minimization
  methods for nonsmooth optimization.
\newblock {\em SIAM Journal on Optimization}, 23(2):1126--1153, 2013.

\bibitem{Ghadimi15acc:nonconvex}
S.~Ghadimi and G.~Lan.
\newblock Accelerated gradient methods for nonconvex nonlinear and stochastic
  programming.
\newblock {\em Mathematical Programming}, pages 1--41, 2015.

\bibitem{Ghadimi14mini}
S.~Ghadimi, G.~Lan, and H.~Zhang.
\newblock Mini-batch stochastic approximation methods for nonconvex stochastic
  composite optimization.
\newblock {\em Mathematical Programming}, pages 1--39, 2014.

\bibitem{scutari13decomposition}
G.~Scutari, F.~Facchinei, P.~Song, D.~P. Palomar, and J.-S. Pang.
\newblock Decomposition by partial linearization: Parallel optimization of
  multi-agent systems.
\newblock {\em IEEE Transactions on Signal Processing}, 63(3):641--656, 2014.

\bibitem{Bolte14}
J.~Bolte, S.~Sabach, and M.~Teboulle.
\newblock Proximal alternating linearized minimization for nonconvex and
  nonsmooth problems.
\newblock {\em Mathematical Programming}, 146, 2014.

\bibitem{Wei13}
E.~Wei and A.~Ozdaglar.
\newblock On the {O}(1/k) convergence of asynchronous distributed alternating
  direction method of multipliers.
\newblock 2013.
\newblock Preprint, available at arXiv:1307.8254.

\bibitem{chang14}
T.-H. Chang.
\newblock A proximal dual consensus admm method for multi-agent constrained
  optimization.
\newblock 2014.
\newblock Preprint, available at arXiv:1409.3307.

\bibitem{bertsekas96}
D.~P. Bertsekas and J.~N. Tsitsiklis.
\newblock {\em Neuro-Dynamic Programming}.
\newblock Athena Scientific, Belmont, MA, 1996.

\end{thebibliography}

\end{document}